\documentclass[letterpaper]{amsart}

\usepackage[margin=1.25in]{geometry}

\usepackage{marginnote}

\usepackage{amsfonts}
\usepackage{euscript}

\usepackage{latexsym,epsfig,verbatim}
\usepackage{amsmath,amsthm,amssymb}
\usepackage{colonequals}
\usepackage{units}

\usepackage{url}
\usepackage{color}
\usepackage[colorlinks,citecolor=blue,linkcolor=red!80!black]{hyperref}
\usepackage[nameinlink]{cleveref}
\newcommand{\crefrangeconjunction}{--}
\usepackage{comment}
\usepackage{amscd}

\usepackage{accents}

\usepackage{tikz}
\usetikzlibrary{matrix,arrows,decorations.pathmorphing}

\theoremstyle{plain}
\newtheorem{theorem}{Theorem}[section]
\newtheorem{lemma}[theorem]{Lemma}
\newtheorem{corollary}[theorem]{Corollary}
\newtheorem{proposition}[theorem]{Proposition}

\newtheorem{prop-def}[theorem]{Proposition-Definition}
\crefname{prop-def}{Proposition-Definition}{Proposition-Definition}
\newtheorem*{main:uniform_finiteness}{\Cref{thm:main}}
\newtheorem*{main:quotient}{\Cref{cor:quotient1}}
\newtheorem*{main:rationalfibs}{\Cref{th:rational_fibers}}
\newtheorem*{main:ir-rational}{\Cref{th:rational_and_irrational}}
\newtheorem*{main:conical}{\Cref{th:conical}}
\newtheorem*{main:continuity}{\Cref{th:lamination_continuity}}

\crefname{claim}{Claim}{Claims}
\newtheorem*{claim*}{Claim}

\theoremstyle{definition}
\newtheorem{definition}[theorem]{Definition}
\newtheorem{remark}[theorem]{Remark}
\newtheorem{example}[theorem]{Example}

\newtheorem{convention}[theorem]{Convention}
\crefname{convention}{Convention}{Conventions}

\newcommand{\define}[1]{\emph{#1}}

\newcommand{\R}{\mathbb{R}}
\newcommand{\Z}{\mathbb{Z}}

\newcommand{\floor}[1]{\left\lfloor#1\right\rfloor}

\newcommand{\norm}[1]{\left\Vert#1\right\Vert}
\newcommand{\abs}[1]{\left\vert#1\right\vert}
\newcommand{\inv}{^{-1}}

\DeclareMathOperator{\Out}{Out}
\DeclareMathOperator{\Aut}{Aut}
\DeclareMathOperator{\Inn}{Inn}
\DeclareMathOperator{\rank}{rank}
\newcommand{\free}{\mathbb{F}} 
\newcommand{\factor}{{\EuScript F}} 
\newcommand{\F}{\factor} 
\renewcommand{\int}{\mathcal{I}}
\newcommand{\fc}{\factor} 
\DeclareMathOperator{\cv}{cv} 
\newcommand{\os}{{\EuScript X}} 
\newcommand{\X}{\os} 
\newcommand{\osbar}{\overline{\os}}

\DeclareMathOperator{\CV}{CV} 
\newcommand{\cvbar}{\overline{\mbox{cv}}}

\newcommand{\dsym}{d^{\mathrm{sym}}_\os} 
\newcommand{\len}{\ell}  

\DeclareMathOperator{\Lip}{Lip} 

\newcommand{\CS}{{\EuScript C}{\EuScript S}}
\newcommand{\DB}{\partial^2 \free }
\renewcommand{\CS}{\mathcal{CS}}

\renewcommand{\L}{\mathcal{L}}

\newcommand{\ind}{\mbox{ind}_\mathcal Q}

\newcommand{\ct}{\partial\iota} 

\newcommand{\I}{\mathbf{I}}

\newcommand{\ca}{\mathrm{Cay}}

\begin{document}
\title[Cannon--Thurston maps for hyperbolic free groups]{\textbf{\Large Cannon--Thurston maps for hyperbolic free group extensions}}
\author{Spencer Dowdall, Ilya Kapovich, and Samuel J. Taylor}
\date{\today}

\begin{abstract}
This paper gives a detailed analysis of the Cannon--Thurston maps associated to a general class of hyperbolic free group extensions. Let $\free$ denote a free groups of finite rank at least $3$ and consider a \emph{convex cocompact} subgroup $\Gamma\le\Out(\free)$, i.e. one for which the orbit map from $\Gamma$ into the free factor complex of $\free$ is a quasi-isometric embedding. The subgroup $\Gamma$ determines an extension $E_\Gamma$ of $\free$, and the main theorem of Dowdall--Taylor \cite{DT1} states that in this situation $E_\Gamma$ is hyperbolic if and only if $\Gamma$ is purely atoroidal.

Here, we give an explicit geometric description of the Cannon--Thurston maps $\partial \free\to\partial E_\Gamma$ for these hyperbolic free group extensions, the existence of which follows from a general result of Mitra. In particular, we obtain a uniform bound on the multiplicity of the Cannon--Thurston map, showing that this map has multiplicity at most $2\rank(\free)$. This theorem generalizes the main result of Kapovich and Lustig \cite{KapLusCT} which treats the special case where $\Gamma$ is infinite cyclic. We also answer a question of Mahan Mitra by producing an explicit example of a hyperbolic free group extension for which the natural map from the boundary of $\Gamma$ to the space of laminations of the free group (with the Chabauty topology) is not continuous.
\end{abstract}

\subjclass[2010]{Primary 20F65, Secondary 57M, 37B, 37D}

\maketitle

\section{Introduction}

A remarkable paper of Cannon and Thurston~\cite{CannonThurston}  proved that if $M$ is a closed hyperbolic 3--manifold which fibers over the circle $\mathbb S^1$ with fiber $S$, then the inclusion of $\widetilde S=\mathbb H^2$ into $\widetilde M=\mathbb H^3$ extends to a continuous $\pi_1(S)$--equivariant surjective map from $\partial \mathbb H^2=\mathbb S^1$ to $\partial \mathbb H^3=\mathbb S^2$. In particular, the inclusion $\pi_1(S)\le\pi_1(M)$ of word-hyperbolic groups extends to a continuous map $\partial\pi_1(S)\to\partial\pi_1(M)$ of their Gromov boundaries.
Though not published till 2007, this paper \cite{CannonThurston} has been highly influential since its circulation as a preprint in 1984. 
Consequently, if the inclusion $\iota\colon H\to G$ of a word-hyperbolic subgroup $H$ of a word-hyperbolic group $G$ extends to a continuous (and necessarily $H$--equivariant and unique) map $\ct\colon\partial H\to \partial G$, the map $\ct$ came to be called the \emph{Cannon--Thurston map}. 
It is easy to see that the Cannon--Thurston map always exists and is injective in the case that $H$ is a quasiconvex subgroup of $G$;
the above result of Cannon and Thurston \cite{CannonThurston} provided the first nontrivial example of existence of $\ct$ in the non-quasiconvex case.

Later the work of Mitra~\cite{MitraCTmaps-general,MitraCTmaps-trees,Mitra98} showed that the Cannon--Thurston map exists in several general situations corresponding to non-quasiconvex subgroups. 
In particular,  Mitra proved \cite{MitraCTmaps-general} that whenever 
\begin{equation}
\label{eqn:hyp_SES}
1 \longrightarrow H \longrightarrow G \longrightarrow \Gamma \longrightarrow 1 
\end{equation}
is a short exact sequence of three infinite word-hyperbolic groups, then the Cannon--Thurston map $\ct\colon\partial H\to\partial G$ exists and is surjective. Only recently did the work of Baker and Riley~\cite{BakerRiley} produce the first example of a word-hyperbolic subgroup $H$ of a word-hyperbolic group $G$ for which the inclusion $H\le G$ does not extend to a Cannon--Thurston map.
Analogs and generalizations of the Cannon--Thurston map have been studied in many other contexts, see for example~\cite{Klar,McM,Miyachi,LLR,LMS,Gerasimov,Bow07,Bow13,MP11,MjD14,Mj14,JKLO}. 
The best understood case concerns discrete isometric actions of surface groups on $\mathbb H^3$, where the most general results about Cannon--Thurston maps are due to Mj~\cite{Mj14}.  

The main results from the theory of JSJ decomposition for word-hyperbolic groups (see \cite{RipsSela} for the original statement and \cite{Levitt} for a clarified version) imply that if we have a short exact sequence \eqref{eqn:hyp_SES} of three infinite word-hyperbolic groups with $H$ being torsion-free, then $H$ is isomorphic to a free product of surface groups and free groups. Thus understanding the structure of the Cannon--Thurston map for such short exact sequences requires first studying in detail the cases where $H$ is a surface group or a free group. The case of word-hyperbolic extensions of closed surface groups is closely related to the theory of convex cocompact subgroups of mapping class groups, and the structural properties of the Cannon--Thurston map in this setting are by now well understood; \cite{LMS} for details. 

In this paper we consider the case of hyperbolic extensions of free groups and specifically the class of hyperbolic extensions introduced by Dowdall and Taylor \cite{DT1}. 
To describe this class, we henceforth fix a free group $\free$ of finite rank at least $3$. Following Hamenst\"adt and Hensel \cite{HaHe}, we say that a finitely generated subgroup $\Gamma\le\Out(\free)$ is \emph{convex cocompact} if the orbit map $\Gamma\to\F$ into the free factor graph $\F$ of $\free$ is a quasi-isometric embedding. 
Hyperbolicity of $\F$ \cite{BF14} and the definition of $\F$ respectively imply convex cocompact subgroups of $\Out(\free)$ are word-hyperbolic and that their infinite-order elements are all fully irreducible.
We say that a subgroup $\Gamma\le\Out(\free)$ is \emph{purely atoroidal} if every infinite-order element $\phi\in \Gamma$ is atoroidal (that is, no positive power of $\phi$ fixes a nontrivial conjugacy class in $\free$).  
Given any subgroup $\Gamma\le\Out(\free)$, the full pre-image $E_\Gamma$ of $\Gamma$ under the quotient map $\Aut(\free)\to\Out(\free)$ fits into a short exact sequence
\begin{equation}
\label{eqn:free_extension}
 1\longrightarrow \free\longrightarrow E_\Gamma\longrightarrow \Gamma\longrightarrow 1
\end{equation}
with kernel $\free\cong \Inn(\free)$.
The main result of \cite{DT1} proves that $E_\Gamma$ is hyperbolic whenever $\Gamma\le \Out(\free)$ is convex cocompact and purely atoroidal. 
From this, we see that if $\Gamma \le \Out(\free)$ is convex cocompact, then $E_\Gamma$ is hyperbolic if and only if $\Gamma$ is purely atoroidal.
Moreover, there is a precise sense \cite{TT} in which random finitely generated subgroups of $\Out(\free)$ 
satisfy these hypotheses 
and so define hyperbolic extensions as in \eqref{eqn:free_extension}. 

Our goal in this paper is to understand the Cannon--Thurston maps for hyperbolic extensions $E_\Gamma$ of $\free$ corresponding to 
convex cocompact subgroups $\Gamma$ of $\Out(\free)$.
Such extensions vastly generalize the hyperbolic free-by-cyclic groups with fully irreducible monodromy whose Cannon--Thurston maps were explored in detail by Kapovich and Lustig in \cite{KapLusCT}.
Our analysis extends many results of \cite{KapLusCT} and gives an explicit description of the Cannon--Thurston map in the general setting of purely atoroidal convex cocompact $\Gamma$.
Moreover, this description reveals quantitative global and local
features of the map and allows us to address multiple conjectures
regarding Cannon--Thurston maps in this setting. In the
case of groups $E_\Gamma$ corresponding to purely atoroidal convex
cocompact $\Gamma\le \Out(\free)$, we answer a question of Swarup (which appears as Question $1.20$ on Bestvina's list) by establishing uniform finiteness of the fibers of $\ct\colon \partial \free \to \partial E_\Gamma$.

\begin{main:uniform_finiteness}
Let $\Gamma\le \Out(\free)$ be purely atoroidal and convex cocompact, where $\free$ is a free group of finite rank at least $3$, and let $\ct\colon\partial \free\to\partial E_\Gamma$ denote the Cannon--Thurston map for the hyperbolic $\free$--extension $E_\Gamma$. Then for every $y\in \partial E_\Gamma$, the degree $\deg(y) = \#\left((\ct)^{-1}(y)\right)$ of the fiber over $y$ satisfies
\[1 \le \deg(y) \le 2\rank(\free).\]
In  particular, the fibers Cannon--Thurston map are all finite and of uniformly bounded size.
\end{main:uniform_finiteness}

To establish \Cref{th:main_1}, we relate Mitra's theory of ``ending laminations'' (\Cref{def:Mitra}) for hyperbolic group extensions \cite{MitraEndingLams} to the theory of algebraic laminations on free groups developed by Coulbois, Hilion, and Lustig \cite{CHL1, CHL2}.  A general result of Mitra~\cite{MitraEndingLams} about Cannon--Thurston maps for short exact sequences of hyperbolic groups (\Cref{th:Mj_lam} below) implies that distinct points $p,q\in \partial \free$ are identified by the Cannon--Thurston map $\ct\colon \partial \free\to\partial E_\Gamma$ if and only if $(p,q)$ is a leaf of an ending lamination $\Lambda_z$ on $\free$ for some $z\in \partial \Gamma$.
Since our $\Gamma$ is convex cocompact by assumption, the orbit map $\Gamma\to\fc$ into the free factor complex extends to a continuous embedding $\partial\Gamma\to\partial\fc$.
By Bestvina--Reynolds \cite{bestvina2012boundary} and Hamenst\"adt \cite{hamenstadt2013boundary},  the boundary $\partial \F$ consist of equivalence classes of arational $\free$--trees. Thus to each $z\in \partial \Gamma$ we may also associate a class $T_z$ of arational $\free$--trees.  
The tree $T_z$ moreover comes equipped with a \emph{dual lamination} $L(T_z)$, as introduced in \cite{CHL2}. Informally, $L(T_z)$ consists of lines in the free group which project to bounded diameter sets in the tree $T_z$. 
Our key technical result, \Cref{th:laminations_agree}, shows that for every $z\in\partial \Gamma$ we have $\Lambda_z=L(T_z)$. Combining this with Mitra's general theory \cite{MitraEndingLams}, we obtain the following explicit description of the Cannon--Thurston map:

\begin{main:quotient}
Let $\Gamma\le\Out(\free)$ be convex cocompact and purely atoroidal. 
Then the Cannon--Thurston map $\ct\colon \partial \free \to \partial E_\Gamma$ identifies points $a,b \in \partial \free$  if and only if there exists $z \in \partial \Gamma$ such that $(a,b) \in L(T_z)$.  That is, $\ct$ factors through the quotient of $\partial \free$ by the equivalence relation\[a\sim b \iff (a,b)\in L(T_z)\text{ for some }z\in \partial \Gamma\]
and descends to an $E_\Gamma$--equivariant homeomorphism $\nicefrac{\partial\free}{\sim} \to \partial E_\Gamma$. 
\end{main:quotient}

We derive \Cref{th:main_1} from \Cref{th:main_2} by using results of Coulbois--Hilion \cite{CoulboisHilion-index} concerning the \emph{$\mathcal Q$--index} for very small minimal actions of $\free$ on $\mathbb{R}$--trees. The point is that the laminations $\Lambda_z$ that Mitra constructs in \cite{MitraEndingLams} are a priori complicated and unwieldy objects from which it is difficult to extract information,
whereas the laminations $L(T_z)$ appearing in \Cref{th:main_2} are subject to the general theory $\R$--trees.
The equality $\Lambda_z=L(T_z)$ provided by \Cref{th:laminations_agree} thus allows us to access this theory and use it to analyze the Cannon--Thurston maps. 

However, establishing the equality $\Lambda_z = L(T_z)$ is nontrivial even in the special case, treated by Kapovich and Lustig \cite{KapLusCT}, when $\Gamma = \langle\phi\rangle$ is a cyclic group generated by an atoroidal fully irreducible element $\phi$. The general case considered here is considerably harder since our trees $T_z$ no longer enjoy the ``self-similarity'' properties of stable trees of atoroidal fully irreducibles.
 The laminations $\Lambda_z$ and $L(T_z)$ are defined in very different terms, and the main difficulty is in establishing the inclusion $\Lambda_z\subseteq L(T_z)$. 
The key step in this direction is \Cref{prop:limit_to_zero} which shows that if $g_i\in\Gamma$ is a quasigeodesic sequence converging to $z\in\partial\Gamma$, then $\ell_{T_z}(g_i(h))\to 0$ for every nontrivial $h\in \free$.
Note that in this situation it is fairly straightforward to see that the projective geodesic current $[\mu]=\lim_{i\to\infty} [\eta_{g_i(h)}]$ satisfies $\langle T_z, \mu\rangle=0$ (where $\langle\cdot,\cdot\rangle$ is the intersection pairing constructed in \cite{kapovich2007geometric}), but this is much weaker than the needed conclusion $\lim_{i\to\infty} \ell_{T_z}(g_i(h))=0$. 
The proof of \Cref{prop:limit_to_zero} relies on recent results of Dowdall and Taylor~\cite{DT1} about folding paths in Culler and Vogtmann's Outer space $\X$ that remain close to the orbit of a purely atoroidal convex cocompact subgroup $\Gamma$.

Our \Cref{th:laminations_agree}, establishing that for every
$z\in\partial \Gamma$ we have $\Lambda_z=L(T_z)$, has quickly found
useful applications in a new paper of Mj and Rafi~\cite{MjRafi}
regarding quasiconvexity in the context of hyperbolic group
extensions. See Proposition~4.3 in~\cite{MjRafi} and its applications
in Theorem~4.11 and Theorem~4.12 of ~\cite{MjRafi}. We remark that the
quasiconvexity result given by Theorem~4.12 of \cite{MjRafi} is also
proved by different methods in the forthcoming paper \cite{DT2}.

\subsection*{Rational and essential points}
In addition to \Cref{th:main_1} and \Cref{th:main_2}, we obtain fine information about the Cannon--Thurston map in regards to rational and essential points.
Recall that a point $\xi$ in the boundary $\partial G$ of a word-hyperbolic group $G$ is called \emph{rational} if 
there is an infinite-order element $g\in G$ such that $\xi$ equals the limit $g^\infty$ in $G\cup \partial G$ of the sequence $\{g^n\}$.
For the short exact sequence \eqref{eqn:free_extension}, a point $y\in \partial E_\Gamma$ is called \emph{$\Gamma$--essential} if there exists a (necessarily unique) point $\zeta(y)\in \partial\Gamma$ such that $y=\ct(p)$ for a point $p\in\partial\free$ that is \emph{proximal} for $L(T_{\zeta(y)})$ in the sense of \Cref{defn:proximal} below.  Informally, $\Gamma$--essential points are the $\ct$--images of points in $\partial\free$ that ``remember'' in an essential way the lamination $L(T_z)$ for some $z\in\partial\Gamma$. 

We write $\deg(y)\colonequals \#\left((\ct)\inv(y)\right)$ for the cardinality of the Cannon--Thurston fiber over $y\in\partial E_\Gamma$ (so $1\le \deg(y)\le 2\rank(\free)$ by \Cref{th:main_1}).
Every $y\in \partial E_\Gamma$ with $\deg(y)\ge 2$ is $\Gamma$--essential and moreover has Cannon--Thurston fiber given by $(\ct)^{-1}(y)=\{p\}\cup\{q\in \partial \free \mid (p,q)\in L(T_{\zeta(y)})\}$ for \emph{every} $p\in (\ct)\inv(y)$ (\Cref{lem:fiber_description}).
However, there are also may $\Gamma$--essential points with $\deg(y)=1$. Our next result describes the fibers of $\ct$ over rational points of $\partial E_\Gamma$.

\begin{main:rationalfibs}
Suppose that $1 \to \free \to E_\Gamma \to \Gamma \to 1$ is a hyperbolic extension with $\Gamma \le \Out(\free)$ convex cocompact. 
Consider a rational point $g^\infty\in \partial E_\Gamma$, where $g\in E_\Gamma$ has infinite order.
\begin{enumerate}
\item Suppose that $g^k$ is equal to $w\in \free\lhd E_\Gamma$ for some $k\ge 1$ (i.e., $g$ projects to a finite order element of $\Gamma$). Then $(\ct)^{-1}(g^\infty) = \{w^\infty\}\subset \partial \free$ and so $\deg(g^\infty) = 1$.
\item Suppose that $g$ projects to an infinite-order element $\phi\in\Gamma$. Then there exists $k\ge 1$ such that the automorphism $\Psi\in\Aut(\free)$ given by $\Psi(w) = g^kwg^{-k}$ is forward rotationless (in the sense of \cite{FH11,CoulboisHilion-botany}) and its set $\mathrm{att}(\Psi)$ of attracting fixed points in $\partial \free$ is exactly $\mathrm{att}(\Psi)=(\ct)\inv(g^\infty)$. Moreover, $g^\infty$ is $\Gamma$--essential and $\zeta(g^\infty) = \phi^\infty$.
\end{enumerate}
\end{main:rationalfibs}

In the case of a cyclic group $\langle\phi\rangle$ generated by an atoroidal fully irreducible automorphism $\phi$, Kapovich and Lustig \cite{KapLusCT} showed that every point $y\in\partial E_{\langle\phi\rangle}$ with $\deg(y)\ge 3$ is rational. We show that when $\Gamma$ is nonelementary this conclusion no longer holds and rather that, with some unavoidable exceptions, rational points in $\partial E_\Gamma$ come in a specific way from rational points in $\partial\Gamma$:

\begin{main:ir-rational}
Suppose that $1 \to \free \to E_\Gamma \to \Gamma \to 1$ is a hyperbolic extension with $\Gamma \le \Out(\free)$ convex cocompact. Then the following hold:
\begin{enumerate}
\item If $y\in \partial E_\Gamma$ has $\deg(y)\ge 3$ and $\zeta(y)\in\partial \Gamma$ is rational, then $y$ is rational.
\item If $y\in \partial E_\Gamma$ has $\deg(y)\ge 2$ and $\zeta(y)\in\partial\Gamma$ is irrational, then $y$ is irrational. 
\end{enumerate}
\end{main:ir-rational}

\subsection*{Conical limit points}
Recall that a point $\xi$ in the boundary $\partial G$ of a word-hyperbolic group $G$ is a \emph{conical limit point} for the action of a subgroup $H\le G$ on $\partial G$ if there exists a geodesic ray in the Cayley graph of $G$ that converges to $\xi$ and has a bounded neighborhood that contains infinitely many elements of $H$. Combining the results of this paper with the results of \cite{JKLO}, we obtain the following:

\begin{main:conical}
Let $\Gamma\le\Out(\free)$ be purely atoroidal and convex cocompact. If $y\in \partial E_\Gamma$ is $\Gamma$--essential, then $y$ is not a conical limit point for the action of $\free$ on $\partial E_\Gamma$. 
In particular, if $\deg(y)\ge 2$ or if $y=g^\infty$ for some $g\in E_\Gamma$ projecting to an infinite-order element of $\Gamma$, then $y$ is not a conical limit point for the action of $\free$.
\end{main:conical}

It is known (see \cite{Gerasimov,JKLO}) in a very general convergence group situation that if a Cannon--Thurston map exists then every conical limit point has exactly one pre-image under the Cannon--Thurston map; thus points with $\ge 2$ pre-images cannot be conical limit points. However, \Cref{cor:main_4} also applies to many $\Gamma$--essential points $y\in \partial E_\Gamma$  with $\deg(y)=1$.

\subsection*{Discontinuity of ending laminations}
In \cite{MitraEndingLams}, Mitra asks whether the map which associates to each point $z\in\partial \Gamma$ the corresponding ending lamination $\Lambda_z$ is continuous with respect to the Chabauty topology on the space of laminations. Of course, in the case of extensions by $\Z$ there is nothing to check since the boundary $\partial \Z$ is discrete. In \Cref{sec:continuity}, we answer Mitra's question in the negative by producing a hyperbolic extension $E_\Gamma$ for which the map $z \mapsto \Lambda_z$ is not continuous. This is done explicitly in \Cref{ex:disc}.

Besides establishing this discontinuity, we also provide a positive result about subconvergence of ending laminations. For the statement, let $\L(\free)$ denote the space of laminations on $\free$ equipped with the Chabauty topology (recalled in \Cref{def:Chabauty}) and let $\Lambda_z$ denote Mitra's \cite{MitraEndingLams} ending lamination for $z\in\partial \Gamma$ (see \Cref{def:Mitra}). For a lamination $L \in \L(\free)$, the notation $L'$ denotes the set of accumulation points of $L$, in the usual topological sense.

\begin{main:continuity}
Let $\Gamma\le\Out(\free)$ be purely atoroidal and convex cocompact, 
and let $\Lambda_z\in\L(\free)$ denote the ending lamination associated to $z\in \partial \Gamma$. Then for any sequence $z_i$ in $\partial \Gamma$ converging to $z$ and any subsequence limit $L$ of the corresponding sequence $\Lambda_{z_i}$ in $\L(\free)$, we have
\[
\Lambda'_z \subset L \subset \Lambda_z.
\]
\end{main:continuity}

\noindent This result can be viewed as a statement about the map $\partial \Gamma\to\L(\free)$, given by $z\mapsto \Lambda_z$, possessing a weak form of continuity.

\subsection*{Acknowledgments}

We are grateful to Arnaud Hilton for providing us a copy of the draft 2006 preprint \cite{CHLunpub} and to Chris Leininger for useful conversations. We also thank the referee for helpful comments.

The first author was partially supported by NSF grant DMS-1204814.
The second author was partially supported by the
NSF grant DMS-1405146. The third author was partially supported by NSF grant DMS-1400498. 

\section{Cannon--Thurston maps}\label{sec:ct-intro} 
In this section, we recall some facts about Cannon--Thurston maps for general hyperbolic extensions. For a word-hyperbolic group $G$, we denote its Gromov boundary by $\partial G$. The following result establishes the existence of the Cannon--Thurston map:

\begin{prop-def}[Mitra \cite{MitraCTmaps-general}]\label{pd:mitra}
Suppose that $1 \to H \to G \to \Gamma \to 1$ is an exact sequence of word-hyperbolic groups. Then the inclusion $\iota\colon H \to G$ admits a continuous extension $\hat\iota\colon H \cup \partial H \to G \cup \partial G$ with $\hat\iota(\partial H)\subseteq \partial G$.
The restricted map $\ct\colonequals\hat\iota|_{\partial H}\colon \partial H \to \partial G$ is called the \define{Cannon--Thurston map} for the inclusion $H \to G$; it is surjective whenever $H$ is infinite.
\end{prop-def}

For an element $g\in G$ denote by $\Phi_g$ the automorphism $h\mapsto
ghg^{-1}$ of $H$. We denote by $\phi_g\in\Out(H)$ the outer automorphism class of $\Phi_g$. Similarly, for an element $q\in \Gamma$ denote by
$\phi_q\in \Out(H)$ the outer automorphism class of $\Phi_g$, where
$g\in G$ is any element that maps to $q$; note that the class $\phi_q$ is independent of the chosen lift $g$.  For a conjugacy class $[h]$ in $H$ and an element $q\in \Gamma$ we
also write $[q(h)]:=[\phi_q(h)]=[ghg^{-1}]$, where $g\in G$ is any
element projecting to $q$.

By construction, the Cannon--Thurston map $\ct\colon \partial H \to \partial G$ in \Cref{pd:mitra} is $H$--equivariant with respect to the left translation actions of $H$ on $\partial H$ and $\partial G$. However, $\ct$ actually turns out to be $G$--equivariant with respect to the action $G \curvearrowright \partial H$ defined by $g\cdot p \colonequals \Phi_g(p)$ for $g\in G$ and $p\in \partial H$. 
Notice that the restricted action $H\curvearrowright \partial H$, namely $h\cdot p = \Phi_h(p)$, agrees with the usual action of $H$ on $\partial H$ by left translation.

\begin{proposition}\label{prop:action}
Suppose $1\to H \to G \to \Gamma\to 1$ is an exact sequence of hyperbolic groups.
Then the map $(g,p)\mapsto g\cdot p$ defines an action of $G$ on $\partial H$ by homeomorphisms. Moreover, the Cannon--Thurston map  $\ct\colon \partial H \to \partial G$ is $G$--equivariant.
\end{proposition}
\begin{proof}
While this is implicit in \cite{KapLusCT}, we include a proof for completeness.
The fact that $(g,p)\mapsto g\cdot p$ defines a group action by homeomorphisms follows directly from the definitions.
Choose $p\in \partial H$ and $g\in G$. To prove $G$--equivariance we must show that $\ct(\Phi_g(p))=g\cdot \ct(p)$. 

Choose a sequence $h_n\in H$ such that $h_n\to p$ in the topology of $H\cup \partial H$. By definition of $\ct$ it follows
 that $h_n\to \ct(p)$ in the topology of $G\cup \partial G$. In $G$ we have $gh_ng^{-1}=\Phi_g(h_n)$ so that 
\[
g\cdot \ct(p)=\lim_{n\to\infty} gh_n=\lim_{n\to\infty} gh_ng^{-1}=\lim_{n\to\infty} \Phi_g(h_n)
\]
 in the topology of $G\cup \partial G$. But definition of $\ct$ the last limit above is exactly $\ct(\Phi_g (p))$. 
\end{proof}

\section{Background on free groups, laminations and $\mathcal Q$--index}

For the entirety of this section let $\free$ be a free group of finite rank $N\ge 2$. We will also fix a free basis $X$ of $\free$ and the Cayley graph $\ca(\free,X)$ of $\free$ with respect to $X$.

\subsection{Laminations on free groups}\label{sec:laminations_on_free}

We denote $\DB:=\{(p,q)\in \partial\free \times \partial\free \colon p\ne q\}$ and endow $\DB$ with the subspace topology from the product topology on $\partial\free \times \partial\free$. There is a natural diagonal left action of $\free$ on $\DB$ by left translations: $w(p,q)\colonequals(wp,wq)$ where $w\in\free$ and $(p,q)\in\DB$. 

An \emph{algebraic lamination} on $\free$ is a closed $\free$--invariant subset $L\subseteq \DB$ such that $L$ is also invariant with respect to the ``flip map'' $\DB\to\DB$, $(p,q)\mapsto (q,p)$.  If $L$ is an algebraic lamination on $\free$, an element $(p,q)\in L$ is also referred to as a \emph{leaf} of $L$.  In the Cayley graph $\ca(\free,X)$ of $\free$ with respect to the free basis $X$,  every leaf $(p,q)\in L$ is represented by a unique unparameterized bi-infinite geodesic $l$ from $p$ to $q$ in $\ca(\free,X)$. In this situation we will also sometimes say that $l$ is a leaf of $L$.   We refer the reader to \cite{CHL1,CHL2} for the background information on algebraic laminations.  

We say that a subset $L\subseteq \DB$ is \emph{diagonally closed} if whenever $p,q,r\in \partial\free$ are three distinct points such that $(p,q),(q,r)\in L$ then $(p,r)\in L$.

An important class of laminations are those corresponding to conjugacy classes of $\free$. For $g \in \free \setminus \{1\}$, we denote by $g^{+\infty} \in \partial \free$ the unique forward limit of the sequence $(g^n)_{n\ge1}$ in $\free \cup \partial \free$. Define $g^{-\infty}$ similarly and note that $(g^{-1})^{+\infty} = g^{-\infty}$. Then define the algebraic lamination
\[
L(g) = \free \cdot (g^{+\infty}, g^{-\infty}) \cup \free \cdot (g^{-\infty},g^{\infty}).
\]
Note that $L(g)$ depends only on the conjugacy class of $g$. Moreover, $L(g)$ is indeed a closed subset of $\DB$ and so is a bona fide algebraic lamination. In what follows, for a subset $A$ of a topological space, we denote the closure of $A$ by $\overline{A}$ and its set of accumulation points by $A'$. For a collection $\Omega$ of conjugacy classes of $\free$, we let $L(\Omega)$ denote the smallest algebraic lamination containing $L(g)$ for each $g\in\Omega$. We observe that 
\begin{align}
L(\Omega) = \overline{\bigcup_{g \in \Omega}L(g)}
\end{align}
Since each $L(g)$ is itself closed in $\DB$, we see that the above closure is unnecessary when $\Omega$ is finite.

Finally, we denote the set of all laminations of $\free$ by $\L(\free)$, which we consider with the Chabauty topology. We recall the definition of this topology:

\begin{definition}[Topology on $\L(\free)$] \label{def:Chabauty}
Let $Y$ be a locally compact metric space and let $C(Y)$ be the collection of closed subsets of $Y$. The \define{Chabauty topology} on $C(Y)$ is defined as the topology generated by the subbasis consisting of
\begin{enumerate}
\item $\mathcal{U}_1(K) = \{C \in C(Y) : C \cap K = \emptyset \}$ for $K \subset Y$ compact.
\item $\mathcal{U}_2(O) = \{ C \in C(Y): C \cap O \neq \emptyset \}$ for $O \subset Y$ open.
\end{enumerate}
\end{definition}

A geometric interpretation of convergence in the Chabauty topology is stated in \Cref{lem:Cabauty_convergence}; it will be needed in \Cref{sec:continuity}.  Recall that the space $C(Y)$ is always compact \cite{canary2006fundamentals}. Returning to the situation of algebraic laminations of $\free$, we note that $\L(\free)$ is closed in $C(\partial^2 \free)$ and hence is itself compact.  We henceforth consider $\L(\free)$ with the subspace topology and refer to this as the Chabauty topology on $\L(\free)$.

\subsection{Outer space and its boundary}
Outer space, denoted $\cv$ and introduced by Culler--Vogtmann in \cite{CVouter}, is the space of $\free$--marked metric graphs, up to some natural equivalence. A \emph{marked graph} $(G,\phi)$ is a core graph $G$ (finite with no valence one vertices) equipped with a marking $\phi\colon R \to G$, which is a homotopy equivalence from a fixed rose $R$ with $\mathrm{rank}(\free)$ petals to the graph $G$. A \emph{metric} on $G$ is a function $\ell$ assigning to each edge of $G$ a positive real number (its \emph{length}) and we call the sum of the lengths of the edges of $G$ its \emph{volume}. A marked metric graph is a triple $(G,\phi, \ell)$, and Outer space is defined to be set of marked metric graph up to \emph{equivalence}, where $(G_1,\phi_1, \ell_1)$ is equivalent to $(G_2,\phi_2,\ell_2)$ if there is an isometry from $G_1$ to $G_2$ in the homotopy class of the \emph{change of marking} $\phi_2 \circ \phi_1^{-1}\colon G_1 \to G_2$. \emph{Projectivized Outer space} $\X$, also sometimes denoted $\CV$, is then defined to be the subset of $\cv$ consisting of graphs of volume $1$. Although points in $\cv$ are as described above, we will often denote a marked metric graph simply by its underlying graph $G$ suppressing the marking and metric.

Given $G \in \cv$, the marking associated to $G$ allows one to measure the length of a conjugacy class $\alpha$ of $\free$. In particular, there is a unique immersed loop in $G$ corresponding to the homotopy class $\alpha$ which we denote by $\alpha\vert G$. The \emph{length of $\alpha$ in $G$}, denoted $\ell(\alpha\vert G)$, is the sum of the lengths of the edges of $G$ crossed by $\alpha\vert G$, counted with multiplicites. The standard topology on $\cv$ is defined as the smallest topology such that each of the length functions $\ell(\alpha\vert\;\cdot\;)\colon \cv \to \mathbb{R}_+$ is continuous \cite{CVouter,Paulin}. 

Given a point $(G,\phi,\ell)$ in $\cv$, we can define $T$ to be the universal cover of $G$ equipped with a metric obtained by lifting the metric $\ell$ and also equipped with an action of $\free$ on $T$ by covering transformations (where $\free$ and $\pi_1(G)$ are identified via the marking $\phi$). Then $T$ is an $\R$--tree equipped with a minimal free discrete isometric action of $\free$.  Under this correspondence, equivalent marked metric graphs correspond to $\free$--equivariantly isometric $\R$--trees.  This procedure provides an identification between $\cv$ and the space of minimal free discrete isometric actions of $\free$ on $\R$--trees, considered up to $\free$--equivariant isometries.  If $T$ corresponds to $(G,\phi,\ell)$ then for every $w\in \free$ we have $\ell(w\vert G)=\ell_T(w):=\min_{x\in T} d_T(x,wx)$. We will also sometime use the notation $i(T,w)$ to denote the translation length of $w$ in $T$, i.e. $i(T,w) = \ell_T(w)$. This notation refers to the intersection form studied in \cite{kapovich2007geometric}; the details of which are not needed here.

We denote by $\cvbar$ the set of all very small minimal isometric
actions of $\free$ on $\R$--trees, considered up to $\free$--equivariant
isometries. As usual, for $T\in\cvbar$ and $w\in \free$, define the \emph{translation length} of $w$ on $T$ as $\ell_T(w):=\inf_{x\in T} d(x,wx)$. 
 It is known that $\cvbar$ is equal to the closure of $\cv$ with respect to the ``axes topology;'' see \cite{CL95,BF-Outer} for the original proof, and see \cite{Gui98} for a generalization. We denote the projectivization of $\cvbar$ by $\overline{\X} = \X \cup \partial \X$. Hence, $\partial \X$ denotes projective classes of very small minimal actions of $\free$ on $\R$--trees which are not free and simplicial; this is the so-called  boundary of Outer space. We remark that $\overline{\X}$ is compact.

We recall how $\Aut(\free)$ and $\Out(\free)$ act on $\cvbar$.  If $T\in\cvbar$ and $\Phi\in\Aut(\free)$, the tree $T\Phi\in \cvbar$ is defined as follows. As a set and a metric space we have $T\Phi=T$. The action of $\free$ is modified via $\Phi$: for every $x\in T$ and $w\in \free$ we have $w\underset{T\Phi}{\cdot}x=\Phi(w) \underset{T}{\cdot}x$.  This formula defines a right action of $\Aut(\free)$ on $\cvbar$. The subgroup $\Inn(\free)\le \Aut(\free)$ is contained in the kernel of this action and therefore the action descends to a right action of $\Out(\free)$ on $\cvbar$: for $\phi\in\Out(\free)$ and $T\in\cvbar$ we have $T\phi\colonequals T\Phi$, where $\Phi\in\Aut(\free)$ is any automorphism in the outer automorphism class $\phi$. At the level of translation length functions, for $T\in\cvbar$, $w\in\free$ and $\phi\in\Out(\free)$  we have $\ell_{T\phi}(w)=\ell_T(\phi(w))$.  Finally, these right actions of $\Aut(\free)$ and $\Out(\free)$ on $\cvbar$ can be transformed into left actions by putting $\Phi T:= T\Phi^{-1}$, $\phi T:=T \phi^{-1}$ for $T\in \cvbar$, $\phi\in\Out(\free)$ and $\Phi\in\Aut(\free)$.

\subsection{Metric properties of Outer space}

For the applications in this paper, we will need a few facts from the metric theory of Outer space. We refer the reader to
\cite{FMout, BF14, DT1} for details on the relevant background. 

If $T_1=(G_1,\phi_1,\ell_1)$ and $T_2=(G_2,\phi_2,\ell_2)$ are two points in $\cv$,  the \emph{extremal Lipschitz distortion} $\Lip(T_1,T_2)$, also sometimes denoted $\Lip(G_1,G_2)$, is the infimum of the Lipschitz constants of all the Lipschitz maps $f\colon (G_1,\ell_1)\to (G_2,\ell_2)$ that are freely homotopic to the the change of marking $\phi_2\circ \phi_1^{-1}$. If one views $T_1$ and $T_2$ as $\R$--trees, then $\Lip(T_1,T_2)$ is the infimum of the Lipschitz constants among all $\free$--equivariant Lipschitz maps $T_1\to T_2$.  It is known that 
\[\Lip(T_1,T_2)=\max_{w\in \free\setminus\{1\}} \frac{\ell_{T_2}(w)}{\ell_{T_1}(w)}.\]
 For $T_1,T_2\in \X$ we put
\[d_\X(T_1,T_2)\colonequals\log \Lip(T_1,T_2)\]
and call $d_\X(T_1,T_2)$ the \emph{asymmetric Lipschitz distance} from $T_1$ to $T_2$. It is known that $d_\X$ satisfies all the axioms of being a metric on $\mathcal X$ except that $d_\X$ is, in general, not symmetric as there exist $T_1,T_2\in\X$ such that $d_\X(T_1,T_2)\ne d_\X(T_2,T_1)$. Because of this asymmetry, it is sometimes convenient to consider the symmetrization of the Lipschitz metric:
\[\dsym(T_1,T_2) \colonequals d_\X(T_1,T_2) + d_\X(T_2,T_1)\]
which is an actual metric on $\X$ and induces the standard topology \cite{FMout}. For a subset $A \in \X$, we denote by $N_K(A)$ the \emph{symmetric $K$--neighborhood} of $A$, which is the neighborhood  of $A$ considered with the symmetric metric.

It is known that for any $T_1,T_2\in\X$ there exists a unit-speed $d_\X$--geodesic $\gamma\colon [a,b]\to \X$ given by a \emph{standard geodesic} from $T_1$ to $T_2$ in $\X$. Such a geodesic is a concatenation of a \emph{rescaling path}, which only alters the edge lengths of $T_1$, followed by a \emph{folding path}. This geodesic has the property that $\gamma(a)=T_1$, $\gamma(b)=T_2$, $b-a=d_\X(T_1,T_2)$ and that for any $a\le t\le t'\le b$ one has $t'-t=d_\X(\gamma(t),\gamma(t'))$. The folding path $\gamma(s)$ has some additional properties arising from its specific construction. We omit describing these properties for the moment (and refer the reader to \cite{FMout,BF14, DT1} for details), but will use them as needed in our arguments.

If $\I\subseteq \R$ is a (possibly infinite) interval, we also say that $\gamma\colon\I\to\X$ is a folding path if for every finite subinterval $[a,b]\subseteq \I$ the restriction $\gamma|_{[a,b]}\colon[a,b]\to \X$ is a folding path giving a unit speed $d_\X$--geodesic in the above sense. Then $\gamma\colon\I\to\X$ is also a unit-speed $d_\X$--geodesic.

\subsection{Dual laminations of very small trees}\label{sec:dual_laminations}

\begin{definition}[Dual lamination]\label{def:dual_lam}
Let $T\in\cvbar$. For each $\epsilon >0$, let $\Omega^{\le \epsilon}(T)$ denote the collection of $1\ne g \in \free$ with $\ell_T(g) \le \epsilon$. We form the algebraic lamination generated by these $\epsilon$--short conjugacy classes:
\[
L^{\le \epsilon}(T) = L(\Omega^{\le \epsilon}(T))=\overline{\bigcup_{g \in \Omega^{\le \epsilon}(T)} L(g)} \subset \partial^2 \free.
\]
The \define{dual lamination} $L(T)\subseteq \DB$  of $T$ is then defined to be
\[
L(T) := \bigcap_{\epsilon>0}L^{\le \epsilon}(T).
\]
\end{definition}
\begin{remark}\label{rem:L(T)}
Note that $L^{\le \epsilon}(T)$ and $L(T)$ are in fact algebraic laminations on $\free$.
Further, it is well-known~\cite{CHL2} and not hard to show that $L(T)$ consists of all $(p,q)\in\DB$ such that for every $\epsilon>0$ and every finite subword $v$ of the bi-infinite geodesic from $p$ to $q$ in $\ca(\free,X)$ there exists a cyclically reduced word $w$ over $X^{\pm 1}$ with $\ell_T(w)\le \epsilon$ such that $v$ is a subword of $w$.
\end{remark}

In this paper, we will only be concerned with a certain class of trees $T \in \cvbar$:
\begin{definition}[Arational tree]
A tree $T\in\cvbar$ is called \emph{arational} if there does not exist a proper free factor $F$ of $\free$ and an $F$--invariant subtree $Y\subseteq T$ such that $F$ acts on $Y$ with dense orbits.
\end{definition}

In \cite{Rey12} Reynolds obtained a useful characterization of arational trees in different terms. This characterization implies that if $T\in\cvbar$ does not arise as a dual tree to a geodesic lamination on a once-punctured surface, then $T$ is arational if and only if $T$ is ``indecomposable'' (in the sense of~\cite{Gui08}) and $\free$ acts on $T$ freely with dense orbits. In particular, if $\phi\in\Out(\free)$ is an atoroidal fully irreducible, then the stable tree $T_\phi$ (discussed in \Cref{sect:fully-irred} below) is free and arational; see \cite{CoulboisHilion-botany}. 

\subsection{The factor complex and its boundary} \label{sec:factor_complex}
The \emph{free factor complex} of $\free$ (for rank$(\free) \ge 3$) is the complex $\F$ defined as following: vertices of $\F$, are conjugacy classes of free factors of $\free$ and vertices $A_0, \ldots, A_k$ span an $k$--simplex if these classes have nested representatives $A_0 <  \dotsb < A_k$. The complex $\F$ was introduced in \cite{HVff} and has since become a central tool for studying the geometry of $\Out(\free)$. In particular, the following theorem is most important for our purposes.

\begin{theorem}[Bestvina--Feighn \cite{BF14}] \label{th:BFhyp}
The free factor complex $\F$ is Gromov-hyperbolic; moreover, an element $\phi\in\Out(\free)$ acts on $\F$ as a loxodromic isometry if and only if $\phi$ is fully irreducible.
\end{theorem}

A central tool in the proof of \Cref{th:BFhyp} is the coarse Lipschitz projection $\pi\colon \X \to \F$ from Outer space to the factor complex, which is defined by sending $G \in \X$ to the collection
\[
\pi(G) = \{\pi_1(G') : G' \subset G \text{ is a connected, proper subgraph}\} \subset \F^0.
\]
By \cite[Lemma 3.1]{BF14}, $\mathrm{diam}_\F(\pi(G)) \le 4$. Further, there is an $L\ge 0$, depending only on $\rank(\free)$, such that $\pi\colon \X \to \F$ is coarsely $L$--Lipschitz \cite[Corollary 3.5]{BF14}. Moreover~\cite[Theorem 9.3]{BF14}, if $\gamma\colon[a,b]\to\X$ is a folding path, then $\pi(\gamma([a,b]))$ is within a uniform Hausdorff distance (independent of $\gamma$) from any $\fc$--geodesic from $\pi(\gamma(a))$ to $\pi(\gamma(b))$. 

As a hyperbolic space, $\F$ has a Gromov boundary. Let $\mathcal{AT}$ be the subspace of $\partial \X$ consisting of projective classes of arational trees. For $T,T' \in \mathcal{AT}$, define $T\approx T'$ to mean that $L(T)=L(T')$. Thus $\approx$ is an equivalence relation on  $ \mathcal{AT}$. The following theorem computes the boundary of $\F$ 
and will be needed in \Cref{sec:laminations_agree}.

\begin{theorem}[Bestvina--Reynolds \cite{bestvina2012boundary}, Hamenst\"adt \cite{hamenstadt2013boundary}] \label{th:factor_complex_boundary}
The projection $\pi\colon \X \to \F$ has an extension to a map $\partial \pi\colon \mathcal{AT} \to \partial \F$ which satisfies the following properties: 
\begin{itemize}
\item If $(G_i)_{i\ge0} \subset \X$ is a sequence converging in $\overline{\X}$ to $T\in \mathcal{AT}$, then $\pi(G_i) \to \partial \pi (T)$ in $\F \cup \partial \F$. 
\item If  $(G_i)_{i\ge0} \subset \X$ is a sequence converging in $\overline{\X}$ to $T\in \overline{\X} \setminus \mathcal{AT}$, then the sequence $(\pi(G_i))_{i\ge0}$ remains bounded in $\F$.
\end{itemize}
Moreover, if $T \approx T'$ then $\partial \pi (T) = \partial \pi (T')$, and the induced map $(\mathcal{AT} / \approx) \to \partial F$ is a homeomorphism.
\end{theorem}

We also record the following useful statement which follows directly from \cite[Theorem~A]{CHR11}:
\begin{proposition}\label{prop:CHR}
Let $T,T'\in\cvbar$ be free arational trees such that $T\not\approx T'$. Then $L(T)\cap L(T')=\varnothing$. Moreover if $(p,q)\in L(T)$ then there does not exist $q'\in \partial\free$ such that $(p,q')\in L(T')$. 
\end{proposition}

\subsection{The $\mathcal Q$--map and the $\mathcal Q$--index} \label{sec:Q_map}

For a tree $T\in\cvbar$, denote $\hat T\colonequals\overline{T}\cup
\partial T$, where $\overline T$ is the metric completion of $T$ and
$\partial T$ is the hyperbolic boundary of $T$. Note that the
action of $\free$ on $T$ naturally extends to an action of $\free$ on $\hat T$.

For a tree $T\in\cvbar$ with dense $\free$--orbits,
Coulbois, Hilion and Lustig~\cite{CHL2} constructed an $\free$--equivariant
surjective map $\mathcal Q_T\colon \partial \free \to \hat T$. The precise
definition of $\mathcal Q_T$ is not important for our purposes but we
will need the following crucial property of $\mathcal Q_T$:

\begin{proposition}\label{prop:keyQ}\cite[Proposition 8.5]{CHL2}
Let $T\in\cvbar$ be a tree with dense $\free$--orbits. 
Then for distinct points $p,p'\in \partial \free$ we have $\mathcal
Q_T(p)=\mathcal Q_T(p')$ if and only if $(p, p')\in L(T)$.
\end{proposition}

For a tree $T\in\cvbar$ we say that a freely reduced word $v$ over
$X^{\pm 1}$ is an \emph{$X$--leaf segment} for $L(T)$ (or just a
\emph{leaf segment} for $L(T)$) if there exists
$(p,p')\in L(T)$ such that $v$ labels a finite subpath of the
bi-infinite geodesic from $p$ to $p'$ in $\ca(\free,X)$.

\begin{definition}\label{defn:proximal}
A point $p\in \partial \free$ is said to be \emph{proximal} for $L(T)$
if for every $v$ such that $v$ occurs infinitely often as a subword of
the geodesic ray from $1$ to $p$ in $\ca(\free,X)$, the word $v$ is a
leaf segment for $L(T)$.   
\end{definition}

\begin{proposition}\label{prop:prox}
The following hold:
\begin{enumerate}
\item For $T\in\cvbar$ the definition of a proximal points for $L(T)$ does not depend on the free basis $X$.

\item If $T,T'\in \cvbar$ are free arational trees such that there exists a point $p\in\partial \free$ that is proximal for both $L(T)$ and $L(T')$, then $L(T)=L(T')$.
\end{enumerate}
\end{proposition}

\begin{proof}
Part (1) easily follows from the fact that for any two free bases $X_1,X_2$ of $\free$, the identity map $\free\to\free$ extends to a quasi-isometry $\ca(\free,X_1)\to\ca(\free,X_2)$. We leave the details to the reader.

For part (2), suppose that $T,T'\in \cvbar$ are free arational trees such that there exists $p\in\partial \free$ which is proximal for both $L(T)$ and $L(T')$. Therefore for every $n\ge 1$ there exists a freely reduced word over $X^{\pm 1}$ of length $n$ which is a leaf-segment for both $L(T)$ and $L(T')$. By a standard compactness argument it then follows that there exists a point $(p_1,q_1)\in L(T)\cap L(T')$. Therefore by \Cref{prop:CHR} we have $L(T)=L(T')$.
\end{proof}

We will need the following known results about the map $\mathcal Q_T$:

\begin{proposition}\label{prop:Qmap}
Let $T\in\cvbar$ be a free $\free$--tree with dense $\free$--orbits. 
Then the following hold:

\begin{enumerate}
\item[(1)]\cite[Proposition~5.8]{CHL2} If $p\in
\partial\free$ is such that $\mathcal Q_T(p)\in\overline{T}$, then $p$
is proximal for $L(T)$.

\item[(2)] \cite[Proposition~5.2]{CoulboisHilion-index} For every $x\in \partial T$ we have $\#(\mathcal Q_T^{-1}(x))=1$.

\item[(3)] For every $x\in \hat T$ we have $1\le \#(\mathcal
  Q_T^{-1}(x))<\infty$.

\item[(4)] There are only finitely many $\free$--orbits of points $x\in \hat
  T$ with $\#(\mathcal Q_T^{-1}(x))\ge 3$.
\end{enumerate}
\end{proposition}

Associated to the map $\mathcal Q_T$ there is a notion of the $\mathcal
Q$--index of $T$, developed in~\cite{CoulboisHilion-index}.
We will only need the definition and properties of the $\mathcal
Q$--index for the case where $T\in\cvbar$ is a free $\free$--tree with
dense orbits, and
so we restrict our consideration to that context.

\begin{definition}[$\mathcal Q$--index]
Let $T\in\cvbar$ be a free $\free$--tree with dense $\free$--orbits. The \emph{$\mathcal Q$--index} of a point $x\in \hat T$ is defined to be $\ind(x):=\max\{0, -2+\#(\mathcal Q_T^{-1}(x))\}$.
The \emph{$\mathcal Q$--index} of the tree $T$ is then defined as
\[
\ind(T):=\sum \ind(x),
\]
where the summation is taken over the set of representatives of
$\free$--orbits of points $x$ of $\hat T$ with  $\#(\mathcal Q_T^{-1}(x))\ge 3$.
\end{definition}

The main result of \cite{CoulboisHilion-index} is the following:

\begin{theorem}\cite[Theorem 5.3]{CoulboisHilion-index}\label{th:index-bound}
Let $\free$ be a finite-rank free group with $\rank(\free)\ge 3$. Then every free
$\free$--tree $T\in\cvbar$ with dense $\free$--orbits satisfies
\[
{\rm ind}_\mathcal Q(T)\le 2\rank(\free)-2.
\]
\end{theorem}

\subsection{Stable trees of fully irreducibles}\label{sect:fully-irred}

For any fully irreducible $\phi\in \Out(\free)$ there is an associated
\emph{stable tree} $T_\phi\in\cvbar$ with the property that $\phi
T_\phi=\lambda T_\phi$ for some $\lambda>1$. The tree
$T_\phi\in\cvbar$ is uniquely determined by $\phi$, up to multiplying
the metric by a positive scalar, and the projective class
$[T_\phi]\in\osbar$ is the unique attracting fixed point for the left
action of $\phi$ on $\osbar$. The tree $T_\phi$ may be explicitly
constructed from a train-track representative $f\colon G\to G$ of $\phi^{-1}$, and the
``eigenvalue'' $\lambda$ in the equation $\phi
T_\phi=\lambda T_\phi$ is the Perron-Frobenius eigenvalue of the
transition matrix of $f$. 

For any fully irreducible $\phi\in \Out(\free)$ the tree
$T_\phi$ has dense $\free$--orbits and is arational~\cite{Rey12,CoulboisHilion-botany}; if in addition $\phi$ is atoroidal
then the action of $\free$ on $T_\phi$ is free.

Suppose now that $\phi\in \Out(\free)$ is an atoroidal fully
irreducible element, so that $\phi
T_\phi=\lambda T_\phi$ for some $\lambda>1$. Then for every
representative $\Phi\in \Aut(\free)$ of the outer automorphism class
$\phi$ the trees $\Phi T_\phi$ and $\lambda T_\phi$ are
$\free$--equivariantly isometric. The metric completions
$\overline{\Phi T_\phi}$ and $\lambda \overline{T_\phi}$ are thus 
$\free$--equivariantly isometric as well.

Using the definition of $\overline{\Phi T_\phi}$ as an $\free$--tree, it follows that there exists a bijective $\lambda$--homothety $H_\Phi\colon \overline{T_\phi} \to \overline{T_\phi}$ which
\emph{represents} $\Phi$ in the sense that for every $x\in
\overline{T_\phi}$ and every $w\in\free$ we have
\begin{equation}
\label{eqn:homothety-action}
H_\Phi(wx)=\Phi^{-1}(w)H_\Phi(x).
\end{equation}
Moreover, there is a unique point $C(H_\Phi)\in  \overline{T_\phi}$
which is fixed by $H_\Phi$; this point is called the \emph{center} of
$H_\Phi$.

It is known that for every representative $\Phi\in \Aut(\free)$ of
$\phi$ there exists a unique homothety $H_\Phi$ representing $\Phi$ in
the above sense. Moreover, it is also known that the set of
homotheties representing all representatives of $\phi$ in
$\Aut(\free)$ is exactly the set
\[
\{w H_{\Phi_0}| w\in\free\}
\]
where $\Phi_0$ is some representatives of $\phi$ in
$\Aut(\free)$. We refer the reader to ~\cite{KapLus-stabil} for details.

We will need a number of  known results relating homotheties $H_\Phi$ to
the map $\mathcal Q_{T_\phi}$ which are summarized in \Cref{prop:att} below.
Before stating this proposition recall that there is a notion of a
\emph{forward rotationless}, or FR, element of $\Out(\free)$ which
allows one to disregard certain periodicity and permutational
phenomena that otherwise complicate the index theory for
$\Out(\free)$. The notion of an FR element of $\Out(\free)$ was first
introduced by Feighn and Handel~\cite{FH11}. We refer the reader to
Definition~3.2 in \cite{CoulboisHilion-botany} for a precise
definition. For our purposes we only need to know that for every fully
irreducible $\phi\in \Out(\free)$ there exists $k\ge 1$ such that
$\phi^k$ is FR \cite[Proposition~3.3]{CoulboisHilion-botany}. Note that in this case $T_{\phi}=T_{\phi^k}$, $L(T_{\phi})=L(T_{\phi^k})$ and
$\mathcal Q_{T_\phi}=\mathcal Q_{T_{\phi^k}}$. Also if $\phi\in\Out(\free)$ is an FR element then $\phi^m$ is also FR for every $m\ge 1$.

\begin{proposition}\label{prop:att}
Let $\phi\in \Out(\free)$ be an atoroidal fully
irreducible FR element. 
\begin{enumerate}
\item[(1)]\cite[Proposition~3.1]{CoulboisHilion-botany} For every representative $\Phi\in \Aut(\free)$ of
$\phi$, the left action of $\Phi$ on $\partial \free$ has finitely
many fixed points, each of which is either a local attractor or a
local repeller. Moreover, the action of $\Phi$ on $\partial \free$ has
at least one fixed point which is a local attractor, and at least one
fixed point which is a local repeller.

\item[(2)]\cite[Lemma 4.3]{CoulboisHilion-botany} Let $\Phi\in \Aut(\free)$ be a representative of $\phi$, and denote by $att(\Phi)$ the set of all fixed points of $\Phi$ in $\partial \free$ that  are  local attractors. Let $H_\Phi$ be the homothety of $T_\phi$ representing $\Phi$.
Then
\[
\mathcal Q_{T_\phi}(att(\Phi))=C(H_\Phi)\qquad\text{and}\qquad \mathcal Q_{T_\phi}^{-1}(C(H_\Phi))=att(\Phi).
\]
\end{enumerate}

\end{proposition}

\begin{corollary}\label{cor:attr}
Let $\phi\in\Out(\free)$ be an atoroidal fully irreducible and let
 $\Phi\in\Aut(\free)$ be a representative of $\phi$.  Let $p\in
att(\Phi)$. Then:
\begin{enumerate}
\item The point $p$ is proximal for $L(T_\phi)$.
\item If $T\in \cvbar$ is a free arational tree such that $L(T)\ne
  L(T_\phi)$ then there does not exist $p'\in \partial \free$ such
  that $(p',p)\in L(T)$.
\end{enumerate}
\end{corollary}
\begin{proof}
\Cref{prop:att} implies that
$Q_{T_\phi}(p)=C(H_\Phi)\in\overline{T_\phi}$. Therefore by part (1) of
\Cref{prop:Qmap}, $p$ is proximal for $L(T_\phi)$, as required.

We now prove part (2) of the corollary. Let $T\in \cvbar$ be a free arational tree such that $L(T)\ne L(T_\phi)$. Suppose that there exists $p\ne p'\in\partial\free$ such that
$(p',p)\in L(T)$.  If we knew that there exists some $q$ such that $(q,p)\in L(T_\phi)$, then we would obtain a contradiction with \Cref{prop:CHR}. However, under the assumptions made on $p$, it may happen that $Q_{T_\phi}^{-1}(Q_{T_\phi}(p))=\{p\}$, so that a point $q$ with $(q,p)\in L(T_\phi)$ does not exist. Thus we cannot directly appeal to  \Cref{prop:CHR}, and need an additional argument, provided below.

Since $p$ is proximal for $L(T_\phi)$, there exist
$X$--leaf segments $v_n$ for $L(T_\phi)$ with $|v_n|\to\infty$ as
$n\to\infty$, such that each $v_n$ occurs infinitely many times as a
subword in the geodesic ray from $1$ to $p$ in $\ca(\free,X)$.
Since $(p',p)$ is a leaf of $L(T)$, it follows that each $v_n$ is also
a leaf-segment for $L(T)$. 

For each $n\ge 1$ choose a geodesic segment $\gamma_n=[u_n,w_n]$ in
$\ca(\free,X)$  with label $v_n$ and passing through the vertex
$1\in\free$ such that $d_{\ca(\free,X)}(u_n,1)\to\infty$ and $d_{\ca(\free,X)}(1,w_n)\to\infty$ as $n\to\infty$. After passing to a subsequence, we may assume that the segments $\gamma_n$ converge to a bi-infinite geodesic from $u\in \partial\free$ to $w\in \partial \free$.

Since $v_n$ is a leaf-segment for $L(T_\phi)$
there exists a sequence $(s_n,s_n')\in L(T_\phi)$ such that the
geodesic from $s_n$ to $s_n'$ in $\ca(\free,X)$ passes through
$\gamma_n$ for every $n\ge 1$.
Similarly, since $v_n$ is a leaf-segment for $L(T)$, there exists a  $(t_n,t_n')\in L(T)$ such that the
geodesic from $t_n$ to $t_n'$ in $\ca(\free,X)$ passes through
$\gamma_n$ for every $n\ge 1$.
By construction it then follows that
\[
\lim_{n\to\infty} (s_n,s_n')=\lim_{n\to\infty} (t_n,t_n')=(u,w).
\]
Since $L(T_\phi)$ and $L(T)$ are closed in $\partial^2\free$, it
follows that $(u,w)\in L(T_\phi)\cap L(T)$.
However, since $T_\phi$, $T$ are free arational trees with $L(T_\phi)\ne L(T)$, this contradicts the conclusion $L(T_\phi)\cap L(T)=\varnothing$ of \Cref{prop:CHR}.
\end{proof}

\section{Hyperbolic extensions of free groups}

For the duration of this paper, we assume that $\free$ is a finite-rank free group with $\rank(\free)\ge 3$.
Note that if $F_2=F(a,b)$ is free of rank two, then for every $\phi\in\Out(F_2)$ we have $\phi([g])=[g^{\pm 1}]$ where $g=[a,b]$. For this reason if $1\to F_2\to E \to Q\to 1$ is a short exact sequence with $Q$ and $E$ hyperbolic, then $|Q|=[E:F_2]<\infty$. On the other hand, free groups of rank at least $3$ admit many interesting word-hyperbolic extensions, as discussed in more detail below.

\subsection{Subgroups of $\Out(\free)$ and hyperbolic extension of free groups} \label{sec:free_extensions}

We now recall a general class of hyperbolic $\free$--extensions constructed in \cite{DT1}. These hyperbolic extensions are the natural generalization of hyperbolic free-by-cyclic groups with fully irreducible monodromy. For any $\Gamma \le \Out(\free)$ there is an $\free$--extension $E_\Gamma$ obtained from the following diagram:
\begin{align} \label{cd:extension}
\begin{CD}
1 @>>> \free @>i>> \Aut(\free)  @>p>>  \Out(\free) @>>> 1\\
@.          @|                 @AAA                     @AAA            @.   \\
1 @>>> \free @>i>> E_\Gamma  @>p>> \Gamma    @>>> 1 \\	
\end{CD}
\end{align}
We say that $E_\Gamma \colonequals p^{-1}(\Gamma)$ is the $\free$--extension corresponding to $\Gamma$.

Recall that $\phi \in \Out(\free)$ is called \define{atoroidal} if no positive power of $\phi$ fixes a conjugacy class of $\free$. A key result of Brinkmann \cite{Brink} shows that for a cyclic subgroup $\langle \phi\rangle\le \Out(\free)$, the extension $E_{\langle \phi\rangle}$ is word-hyperbolic if and only if $\phi$ is atoroidal or finite order.
A subgroup $\Gamma \le \Out(\free)$ is  said to be \define{purely atoroidal} if every infinite order element of $\Gamma$ is atoroidal. 

The following theorem gives geometric conditions on a subgroup $\Gamma \le \Out(\free)$ that imply the corresponding extension $E_\Gamma$ is hyperbolic:

\begin{theorem}[Dowdall--Taylor \cite{DT1}] \label{th: DT1}
Let $\Gamma \le \Out(\free)$ be finitely generated. Suppose that $\Gamma$ is purely atoroidal and that for some $A \in \F^0$ the orbit map $\Gamma \to \F$ given by $g \mapsto gA$ is a quasi-isometric embedding. Then the corresponding extension $E_\Gamma$  is hyperbolic.
\end{theorem}

Recall that we have called a finitely generated subgroup $\Gamma \le \Out(\free)$ \emph{convex cocompact} if some orbit map $\Gamma \to \F$ is a quasi-isometric embedding. Hence, \Cref{th: DT1} implies that if $\Gamma$ is convex cocompact, then $E_\Gamma$ is hyperbolic if and only if $\Gamma$ is purely atoroidal. Note that if $\phi \in \Out(\free)$ is fully irreducible, then $\langle \phi \rangle$ is convex cocompact by \Cref{th:BFhyp}.

\begin{remark}[Reformulation in terms of the co-surface graph]\label{rmk:reform}
In \cite{DT2}, the authors reformulate \Cref{th: DT1} in terms of the co-surface graph $\CS$. This is the $\Out(\free)$--graph defined as follows: vertices are conjugacy classes of primitive elements of $\free$ and two conjugacy classes $\alpha$ and $\beta$ are joined by an edge whenever there is a once punctured surface $S$ whose fundamental group can be identified with $\free$ in such a way that $\alpha$ and $\beta$ both represent simple closed curves on $S$. We note that closely related graphs appear in \cite{kapovich2007geometric, MR2, Mann-thesis}; see \cite{DT2} for a discussion and further references. In \cite[Theorem 9.2]{DT1}, it is shown that if $\Gamma \le \Out(\free)$ admits a quasi-isometric orbit map into $\CS$, then $\Gamma$ is purely atoroidal and convex cocompact, and hence the corresponding extension $E_\Gamma$ is hyperbolic. In \cite{DT2}, the converse is proven: A finitely generated subgroup $\Gamma \le \Out(\free)$ admits a quasi-isometric orbit map into the co-surface graph if and only if $\Gamma$ is purely atoroidal and convex cocompact. The authors in \cite{DT2} use this characterization to further study the geometry of the hyperbolic extension $E_\Gamma$.
\end{remark}

\subsection{Laminations for hyperbolic extensions} \label{sec:Mitra_lam}
Fix $\Gamma \le \Out(\free)$ finitely generated such that the corresponding extension 
\[
1 \longrightarrow \free \longrightarrow E_\Gamma \longrightarrow \Gamma \longrightarrow 1
\]
is an exact sequence of hyperbolic groups.

\begin{definition}[Mitra's laminations] \label{def:Mitra}
Let $z\in\partial \Gamma$.
Let $\rho$ be a geodesic ray in $\Gamma$ from $1$ to $z$, with the vertex sequence $g_1,g_2,g_3,\dots, g_n,\dots$ in $\Gamma$.

For $1\ne h\in \free$ let $w_n$ be a cyclically reduced word over $X^{\pm 1}$ representing the conjugacy class $[g_n(h)]$ in $\free$. Let $R_{z,h}$ be the set of all pairs $(u,u')\in \free\times\free$ such that the freely reduced form $v$ of $u^{-1}u'$ occurs a subword in a cyclic permutation of $w_n$ or of $w_n^{-1}$ for some $n\ge 1$.
Put
\[
\Lambda_{z,h}=\overline{R_{z,h}}\cap \DB
\] 
where $\overline{R_{z,h}}$ is the closure of $R_{z,h}$ in $(\free\cup\partial\free)\times (\free\cup\partial \free)$.
Thus $\Lambda_{z,h}$ consists of all $(p_1,p_2)\in\DB$ such that there exists a sequence $(u_i,u_i')\in \free\times\free$ converging to $(p_1,p_2)$ in $(\free\cup\partial\free)\times (\free\cup\partial \free)$ as $i\to\infty$ and such that for every $i\ge 1$ the freely reduced form $v_i$ of $(u_i)^{-1}u_i'$ occurs as a subword in a cyclic permutation of some $w_{n_i}$ or of $w_{n_i}^{-1}$ (which, since $p_1\ne p_2$, automatically implies that $n_i\to\infty$ as $i\to\infty$). 
Put 
\[
\Lambda_{z} := \bigcup_{h\in \free\setminus\{1\}}\Lambda_{z,h}.
\]
Finally, define the \emph{ending lamination} of the extension to be
\[
\Lambda := \bigcup_{z\in \partial \Gamma} \Lambda_{z}.
\]
\end{definition}

\begin{remark}\label{rem:Lambda_z}
Thus $\Lambda_{z,h}$ consists of all $(p,q)\in \DB$ such that for every subword $v$ of the bi-infinite geodesic from $p$ to $q$ in $\ca(\free,X)$ there exists $m\ge 1$ such that $v$ is a subword of a cyclic permutation of $w_m$ or $w_m^{-1}$.

Mitra \cite[Lemma 3.3]{MitraEndingLams} shows that the definition of $\Lambda_z$ does not depend on the choice of a geodesic ray $(g_n)_n$ from $1$ to $z$ in the Cayley graph of $\Gamma$.
Moreover, the proof of  \cite[Lemma 3.3]{MitraEndingLams} implies that instead of a geodesic ray one can also use any quasigeodesic sequence from $1$ to $z$ in $\Gamma$. 
\cite[Remark on p. 399]{MitraEndingLams} also shows that for every $z\in\partial \Gamma$ there exists a finite subset $R\subseteq \free\setminus\{1\}$ such that $\Lambda_z=\cup_{h\in R} \Lambda_{z,h}$. Since every $\Lambda_{z,h}$ is an algebraic lamination on $\free$, it follows that $\Lambda_z$ is also an algebraic lamination on $\free$.

The results of Mitra~\cite{MitraEndingLams}  imply that $\Lambda_{z,h}$ does not depend on the choice of a free basis $X$ of $\free$. 
Therefore $\Lambda_z$ and $\Lambda$ are independent of $X$ as well. In \Cref{lem:mitraII} below we give an equivalent definition of $\Lambda_{z,h}$ which does not involve the choice of $X$; this gives another proof that $\Lambda_{z,h}$ is independent of  $X$.

Mitra \cite{MitraEndingLams} in fact defines $\Lambda_{z,h}$, $\Lambda_z$ and $\Lambda$ in the context of an arbitrary short exact sequence of word-hyperbolic groups. As it suffices for our purposes, here we have only presented the definitions in the somewhat more transparent setting free group extensions.
\end{remark}

Recall that for a collection $\Omega$ of conjugacy classes of $\free$, we define $L(\Omega)$ to be the smallest algebraic lamination containing $L(g)$ for each $g\in \Omega$.

\begin{lemma}\label{lem:mitraII}
For $z \in \partial \Gamma$, let $(g_i)_{i\ge0}$ be a geodesic ray in $\Gamma$ such that $\lim_{i\to \infty} g_i =z \in \partial \Gamma$. Then for any $h \in \free \setminus \{1\}$ we have
\[
\Lambda_{z,h} = \bigcap_{k\ge0} L\big(\{g_i (h): i\ge k\}\big). 
\]
\end{lemma}
\begin{proof}
Let $w_n$ be the cyclically reduced form over $X^{\pm 1}$ of $g_n(h)$. 
Recall that $\Lambda_{z,h}$ consists of all $(p,q)\in \DB$ such that for every finite subword $v$ of the bi-infinite geodesic in $\ca(\free,X)$ from $p$ to $q$ there exists $n\ge 1$ such that $v$ is a subword of a cyclic permutation of $w_n$ or of $w_n^{-1}$.

Put $L\colonequals \bigcap_{k\ge0} L(\{g_i (h): i\ge k\})$. Then $L$ consists of all $(p,q)\in \DB$ such that for every finite subword $v$ of the bi-infinite geodesic in $\ca(\free,X)$ from $p$ to $q$ and every $M\ge 1$ there exist $n\ge M$ and $m\in\mathbb Z\setminus\{0\}$ such that $v$ is a subword of a cyclic permutation of $w_n^m$.  Hence $\Lambda_{z,h}\subseteq L$.

 Let $(p,q)\in L$ be arbitrary.  Let $v$ be a finite subword of of the bi-infinite geodesic in $\ca(\free,X)$ from $p$ to $q$. 
 We will use the following claim to complete the proof:
 \begin{claim*} \label{claim:growth}
The cyclically reduced length $\norm{w_n}$ of $w_n$ tends to $\infty$ as $n\to\infty$.
\end{claim*}
Assuming the claim, choose $M\ge 1$ such that for all $n\ge M$ we have $\norm{w_n}\ge\abs{v}$. Since $(p,q)\in L$, there exist $n\ge M$ and $m\in\mathbb Z\setminus\{0\}$ such that $v$ is a subword of a cyclic permutation of $w_n^m$. The fact that $\norm{w_n}\ge \abs{v}$ implies that $v$ is a subword of a cyclic permutation of $w_n$ or of $w_n^{-1}$. Therefore $(p,q)\in \Lambda_{z,h}$. Hence $L\subseteq \Lambda_{z,h}$ and so $L= \Lambda_{z,h}$, as required.

We now prove the claim. Note that since $E_\Gamma$ is hyperbolic, each infinite order element of $\Gamma$ is atoroidal. Hence, if we denote by $\Gamma_\alpha$ the subgroup of $\Gamma$ consisting of those elements which fix the conjugacy class $\alpha$, then $\Gamma_\alpha$ is a torsion subgroup of $\Out(\free)$ and hence by \cite[Lemma 2.13]{DT1} has $\abs{\Gamma_\alpha} \le e$ for some $e\ge 0$ depending only on $\rank(\free)$. If the claim is false, then there is a $D \ge0$ and a infinite subsequence such that $\norm{g_{n_i}(h)}\le D$ for all $i\ge0$. Let $C$ denote the finite number  (depending on $D$ and $X$) of conjugacy classes of $\free$ whose cyclically reduced length is at most $D$. Hence if $k \ge C(e+2)$, we may find at least $e+2$ distinct elements in the list $g_{n_1}(h),\dotsc,g_{n_k}(h)$ that all belong to the same conjugacy class $\alpha$. This produces $e+1$ distinct elements of $\Gamma$ which fix the conjugacy class $\alpha$. This contradicts our choice of $e$ and completes the proof of the claim.
\end{proof}

The main result of Mitra in \cite{MitraEndingLams} is:

\begin{theorem}\label{th:Mj_lam}\cite[Theorem 4.11]{MitraEndingLams}
Suppose that $1 \to H \to G \to \Gamma \to 1$ is an exact sequence of hyperbolic groups with Cannon--Thurston map $\partial i: \partial H \to \partial G$. Then for distinct points $p,q\in \partial H$, $\partial i(p) = \partial i(q)$ if and only if $(p,q)\in \Lambda$ if and only if $(p,q) \in \Lambda_z$ for some $z \in \partial \Gamma$. 
\end{theorem}

\section{Dual laminations at the boundary of $\Gamma$} \label{sec:laminations_agree}

\begin{convention}\label{conv:main}
For the remainder of this paper, we fix a free group $\free$ of finite rank at least $3$ and a finitely generated, purely atoroidal, convex cocompact subgroup $\Gamma\le\Out(\free)$. Thus we may choose a cyclic free factor $x \in \F^0$ so that the orbit map $\Gamma \to \F$ given by $g \mapsto gx$ is a quasi-isometric embedding. This orbit map then induces a $\Gamma$--equivariant topological embedding $\kappa \colon \partial \Gamma \to \partial \F$ and we identify $\partial \Gamma$ with its image in $\partial \F$. Hence, each point $z\in\partial \Gamma$ correspond to equivalence class $T_z$ of arational trees, each of which has a well-defined dual lamination $L(T_z)$ (\Cref{def:dual_lam}).
Furthermore, \Cref{th: DT1} shows that the corresponding extension $E_\Gamma$ is a hyperbolic group. Thus the short exact sequence 
\[1 \longrightarrow \free \longrightarrow E_\Gamma\longrightarrow \Gamma \longrightarrow 1\]
of hyperbolic groups admits a surjective Cannon--Thurston map $\ct\colon \partial \free\to\partial E_\Gamma$ by \Cref{pd:mitra}, and for every $z\in \partial \Gamma$ there is a corresponding ending lamination $\Lambda_z$ as defined by Mitra in \Cref{def:Mitra}
\end{convention}

The main result of this section characterizes the laminations $\{\Lambda_z  \colon  z\in \partial \Gamma \}$ appearing in \Cref{th:Mj_lam} for the extension $1 \to \free \to E_\Gamma \to \Gamma \to 1$. Recall that we have denoted by $\partial \pi :\mathcal{AT} \to \partial \F$ the map which associates to each arational tree of $\cvbar$ the corresponding point in the boundary of the factor complex (see \Cref{th:factor_complex_boundary}).

\begin{theorem} \label{th:laminations_agree}
For each $z \in \partial \Gamma$, there is $T_z \in \cvbar$ which is free and arational such that $z \mapsto \partial \pi(T_z)$ under $\partial \Gamma \to \partial \F$ with the property that
\[\Lambda_z = L(T_z). \]
\end{theorem}

\begin{corollary}\label{cor:quotient1}\label{th:main_2}
Let $\Gamma\le\Out(\free)$ be convex cocompact and purely atoroidal. 
Then the Cannon--Thurston map $\ct\colon \partial \free \to \partial E_\Gamma$ identifies points $a,b \in \partial \free$  if and only if there exists $z \in \partial \Gamma$ such that $(a,b) \in L(T_z)$.  That is, $\ct$ factors through the quotient of $\partial \free$ by the equivalence relation\[a\sim b \iff (a,b)\in L(T_z)\text{ for some }z\in \partial \Gamma\]
and descends to an $E_\Gamma$--equivariant homeomorphism $\nicefrac{\partial\free}{\sim} \to \partial E_\Gamma$. 
\end{corollary}
\begin{proof}
The specified equivalence relation is by definition given by the subset $\bigcup_{z\in \partial \Gamma}L(T_z)=\Lambda$ of $\DB$, where the last equality holds by  \Cref{th:laminations_agree}. \Cref{th:Mj_lam} asserts that $\Lambda=\{(p,q)\in \partial\free\times\partial\free : \ct(p)=\ct(q)\}$.
Since the Cannon--Thurston map $\ct:\partial\free\to \partial E_\Gamma$ is continuous,  it follows that $\Lambda$ is a closed subset of $\partial\free\times\partial\free$. Therefore $\nicefrac{\partial\free}{\sim}$, equipped with the quotient topology, is a compact Hausdorff topological space. Moreover, the continuity and surjectivity of  $\ct\colon\partial\free\to \partial E_\Gamma$ now imply that $\ct$ quotients through to a continuous bijective map  $J\colon\nicefrac{\partial\free}{\sim}\to \partial E_\Gamma$, which is, by construction, $E_\Gamma$--equivariant. The fact that both $\nicefrac{\partial\free}{\sim}$ and $\partial E_\Gamma$ are compact Hausdorff topological spaces implies that $J$ is a homeomorphism, as required.
\end{proof}

Recall that by a general result of \cite{CHL0}, for every $z\in \partial \Gamma$ the map $\mathcal Q_{T_z}\colon \partial\free\to \hat T_z$ quotients through to a $\free$--equivariant homeomorphism $\partial\free/L(T_z)\to \hat T_z$, where $\partial\free/L(T_z)$ is given the quotient topology and where $\hat T_z$ is given the ``observer's topology''. Now a similar argument to the proof of \Cref{cor:quotient1} implies the following statement (we leave the details to the reader):

\begin{corollary}\label{cor:quotient2}
For each $z\in \partial \Gamma$ the Cannon--Thurston map $\ct\colon\partial \free \to E_\Gamma$ factors through $\mathcal Q_z\colon \partial \free \to \hat T_z$ and induces a continuous, surjective $\free$--equivariant map $\hat T_z \to \partial E_\Gamma$ (where $\hat T_z$ is equipped with the observer's topology).
\end{corollary}

We now start working towards the proof of \Cref{th:laminations_agree}. For our next lemma we assume that the reader has some familiarity with folding paths in $\X$; for example \cite[Section 2, 4]{BF14}. This material is also summarized in \cite[Section 2.7]{DT1} and the reader may find helpful the discussion appearing before Lemma $6.9$ of \cite{DT1}.
For a folding path $G_t$, we say that a conjugacy class $\alpha$ is \define{mostly legal} at time $t_0$ if its \emph{legal length} $\mathrm{leg}(\alpha\vert G_{t_0})$ is at least half of its total length $\len(\alpha\vert G_{t_0})$. Of course, if $\alpha$ is mostly legal at time $t_0$, then it is mostly legal for all $t\ge t_0$.  Also, for any path $q\colon \I \to \X$, we say that $q$ has the $(\lambda,N_0)$--flaring property for constants $\lambda, N_0 \ge1$ if for any $t \in \I$ and any $\alpha \in \free \setminus \{1\}$
\[
\lambda \cdot \ell(\alpha| q(t)) \le \max \big\{\ell(\alpha| q(t-N_0)), \ell(\alpha| q(t+N_0))\big\}.
\]

Fix a rose $R \in \X$ with a petal labeled by our fixed $x \in \F^0$. Observe that $x$ is contained in the projection $\pi(R)$ of $R$ to the factor complex $\F$.

\begin{lemma}[Everyone's eventually mostly legal] \label{lem:eventual_growth}
For $\Gamma \le \Out(\free)$ as in \Cref{conv:main} and for any $K \ge0$ and $\lambda >1$, there exist $N_0, c\ge1$ satisfying the following:
Suppose that $\gamma\colon \I \to \X$ is a unit speed folding path contained in a symmetric $K$--neighborhood of $\Gamma \cdot R \subset \X$. Then $\gamma\colon \I \to \X$ has the $(\lambda,N_0)$--flaring property. Moreover, for any conjugacy class $\alpha$ whose length along $\gamma$ is minimized at $t_\alpha\in\I$, we have for all $t\ge t_{\alpha}$
\[
\frac{1}{c} \cdot e^{(t - t_\alpha)}\ell(\alpha|\gamma(t_\alpha)) \le \ell(\alpha|\gamma(t)) \le  e^{(t - t_\alpha)}\ell(\alpha|\gamma(t_\alpha)).
\]
\end{lemma}

\begin{proof}
That $\gamma\colon \I \to \X$ has the $(\lambda,N_0)$--flaring property is exactly the conclusion of Proposition~6.11 of \cite{DT1} (note that the needed ``$A_0$--QCX'' hypothesis follows from \cite[Corollary 6.3]{DT1}). Also, the upper bound in the statement of the lemma follows immediately from the definition of a unit speed folding path (see \cite[Section 4]{BF14}), so we focus on the lower bound.

For the conjugacy class $\alpha$, let $s_\alpha$ be the infimum of times for which $\alpha$ is mostly legal (if such a time does not exist, set $s_\alpha$ to be the right endpoint of $\I$). We show that $s_\alpha - t_\alpha \le C$, for some constant $C$ not depending on $\alpha$. Then \cite[Lemma 6.10]{DT1} (which is an application of \cite[Corollary 4.8]{BF14}) implies that for $t\ge t_\alpha$
\begin{align*}
\ell(\alpha|\gamma(t)) &\ge \frac{1}{3}e^{t -s_\alpha} \mathrm{leg}(\alpha| \gamma(s_\alpha))\\
&\ge \frac{1}{6}e^{t -s_\alpha} \ell(\alpha| \gamma(s_\alpha)) \\
&= \frac{1}{6}e^{t -t_\alpha}e^{-(s_\alpha -t_\alpha)}\ell(\alpha| \gamma(s_\alpha))\\
&\ge \frac{1}{6e^{C}} \cdot e^{t -t_\alpha}\ell(\alpha| \gamma(t_\alpha)),
\end{align*}
as needed. Hence, it suffice to prove the uniform bound $s_\alpha - t_\alpha \le C$ over all nontrivial conjugacy classes $\alpha$. This will follow from applying the flaring property of the folding path $\gamma$; the idea is that if $\alpha$ is not mostly legal at some time $t_0$ then either the length of $\alpha$ decreases at some definite rate at $t_0$ (which is impossible if $t_0 = t_\alpha$), or after a bounded amount of time $\alpha$ becomes mostly legal. The details are slightly technical and our argument relies on the proof of Proposition $6.11$ of \cite{DT1}.

Since the image of $\gamma$ is contained in the $K$--neighborhood (with respect to $\dsym$) of $\Gamma \cdot R$, there is an $\epsilon >0$ depending only on $R\in \X$ and $K\ge 1$ such that $\gamma(\I) \subset \X_{\ge \epsilon}$, the $\epsilon$--thick part of $\X$. The flaring property then implies that there is a $M\ge1$ depending only on $\lambda$ and $\epsilon$ such that 
\begin{align}\label{eq:flare}
12 \le \ell(\alpha | \gamma(t_0 -N_0))\le \frac{1}{\lambda}\ell(\alpha | \gamma(t_0)),
\end{align}
for $t_0 =t_\alpha +M$. Hence, it suffices to bound the difference $s_\alpha -t_0$. According to the proof of Proposition $6.11$ of \cite{DT1} either $(1)$ $\mathrm{ilg}(\alpha|\gamma(t_0)) \ge \frac{\ell(\alpha|\gamma(t_0))}{2}$, or $(2)$ $\mathrm{ilg}(\alpha|\gamma(t_0)) < \frac{\ell(\alpha|\gamma(t_0))}{2}$ and $\mathrm{leg}(\alpha|\gamma(t_0)) > 0$. Here, $\mathrm{ilg}(\alpha|\gamma(t_0))$ is the illegal length of $\alpha$ as defined in Section $6$ of \cite{DT1}. (We note that the third case of \cite[Proposition 6.11]{DT1} does not arise since in that case $\ell(\alpha | \gamma(t_0)) < 6$.) 

In case $(1)$, Proposition $6.11$ shows that $\ell(\alpha|\gamma(t_0 -N_0)) \ge \lambda \cdot \ell(\alpha|\gamma(t_0))$, which directly contradicts \eqref{eq:flare}. Hence, we conclude that we are in the situation of case $(2)$ of \cite[Proposition 6.11]{DT1}, where it is shown that the legal length constitutes a definite fraction of the total length of $\alpha$ in $\gamma(t_0)$. In fact, there it is shown that
\[
\mathrm{leg}(\alpha|\gamma(t_0)) \ge \frac{\ell(\alpha|\gamma(t_0)) -6}{2(1+ \breve{m})}\ge \frac{\ell(\alpha|\gamma(t_0))}{4(1+ \breve{m})},
\]
where $\breve{m}$ is a constant depending only on the rank of $\free$. From this it follows easily that $s_\alpha - t_0$ is uniformly bounded (e.g., \cite[Lemmas 6.9--6.10]{DT1} show that illegal length decays at a definite rate whereas legal length grows at a definite rate). This completes the proof of the lemma.
\end{proof}

The companion to \Cref{lem:eventual_growth} is the following proposition, which states that we can extract folding rays in $\X$ which stay uniformly close to the orbit of $\Gamma$, have the required flaring property, and limit to free, arational trees in $\partial \X$. Most of this follows from the main technical work in \cite{DT1} on stable quasigeodesics in $\X$. Recall we have fixed $R \in \X$ with a petal labeled by $x$.

\begin{proposition}[Folding rays to infinity]\label{lem:folding_rays}
For any $k, \lambda \ge1$ there are $M, K\ge0$ such that if $(g_i)_{i\ge0}$ is a $k$--quasigeodesic ray in $\Gamma$, then there is an infinite length folding ray $\gamma\colon \I \to \X$ parameterized at unit speed with the following properties:
\begin{enumerate}
\item The sets $\gamma(\I)$ and $\{g_iR : i\ge 0\}$ have symmetric Hausdorff distance at most $K$.
\item The rescaled folding path $G_t = e^{-t}\cdot \gamma(t)\in\cv$ converges to the arational tree $T \in \partial\cv$ with the property that $\lim_{i\to \infty}g_i x= \partial \pi(T)$ in $\F \cup \partial \F$, where $\partial \pi(T)$ is the projection of the projective class of $T$ to the boundary of $\F$. Moreover, the action $\free \curvearrowright T$ is free.
\item The folding path $\gamma$ has the $(\lambda,M)$ flaring property.
\end{enumerate}
\end{proposition}

\begin{proof}
By Theorem $5.5$ of \cite{DT1}, the orbit $\Gamma \cdot R$ is quasiconvex; hence, there is a $K\ge0$ depending only on $k\ge0$ (and the quasi-isometry constants of the orbit map $\Gamma \to \F$) such that any geodesic of $\X$ joining points of $(g_iR)_{i\ge0}$ is contained in the symmetric $K$--neighborhood of the quasigeodesic $ (g_iR)_{i\ge0}$ (note that $\Gamma$ is word-hyperbolic and moreover qi-embedded into $\os$ by \cite[Lemma 6.4]{DT1}). Let $\gamma_i$ be a standard geodesic of $\X$ joining $g_0R$ to $g_iR$. Since this collection of geodesics begins at $g_0R$ and remains in a symmetric $K$--neighborhood of $(g_iR)_{i\ge0}$, the Arzela--Ascoli theorem implies that (after passing to a subsequence) the $\gamma_i$ converge uniformly on compact sets to a geodesic ray $\gamma\colon \I \to \X$, which is also contained in the symmetric $K$--neighborhood of $ (g_iR)_{i\ge0}$. Hence, the geodesic $\gamma$ is contained in $\X_{\ge \epsilon}$ for $\epsilon$ depending only on $K$. As in the proof of Lemma $6.11$ of \cite{bestvina2012boundary}, we see that except for some initial portion of $\gamma$ of uniformly bounded size, $\gamma$ is a folding path. Hence, up to increasing $K$ by a bounded amount, this completes the proof of item $(1)$.

To prove $(2)$, let $\xi$ denote the limit of $(g_i x)_{i\ge0}$ in $\partial \F$. Note that the rescaled folding path $G_t = e^{-t}\cdot \gamma(t)$ is isometric on edges and hence converges to a tree $T \in \cvbar$ \cite{HMaxes}. Hence $\gamma(t)$ converges to the projective class of $T$ in $\overline{\X}$ as $t \to \infty$ and by item $(1)$ 
\[
\lim_{t\to \infty} \pi (\gamma(t)) = \lim_{i \to \infty} \pi(g_iR) = \lim_{i \to \infty} g_i x =\xi,
\]
in $\F \cup \partial \F$. Hence, the tree $T$ is arational and $\partial \pi (T)= \xi \in \partial \F$ by \Cref{th:factor_complex_boundary}.
To compete the proof of $(2)$, it only remains to show that the tree $T$ has a free $\free$--action. This will follow using item $(3)$, which we note follows immediately from Proposition $6.11$ of \cite{DT1}.

To see that $\free \curvearrowright T$ is free, it suffices to show that  $\ell_T(\alpha) >0$ for each $\alpha \in \free \setminus \{1\}$. Since $\lim_{t \to \infty} G_t = T$, we see using $(3)$ and \Cref{lem:eventual_growth} that
\begin{align*}
\ell_T(\alpha) &= \lim_{t \to \infty}\ell(\alpha|G_t) \\
&=  \lim_{t \to \infty} e^{-t} \cdot \ell(\alpha|\gamma(t))\\
& \ge \lim_{t \to \infty} e^{-t} \cdot \frac{ e^{(t - t_\alpha)}}{c} \ell(\alpha|\gamma(t_\alpha))  \\
& = \frac{1}{ce^{ t_\alpha}} \cdot  \ell(\alpha|\gamma(t_\alpha))\\
& >0.
\end{align*}
This completes the proof.
\end{proof}

\Cref{lem:folding_rays}, together with the fact that every fully irreducible element of $\Out(\free)$ acts on $\F$ as a loxodromic isometry, immediately implies:

\begin{corollary}\label{cor:phi}
Let $\phi\in\Gamma$ be an element of infinite order. Then for the point $z=\phi^{\infty}\in \partial\Gamma$ we have $T_z=T_\phi$. That is, $z$ is mapped under the map $\partial \Gamma\to\partial\F$ to the $\approx$--equivalence class $[T_\phi] \in \partial \F$, where $T_\phi$ is the stable tree of $\phi$.
\end{corollary}

Most of the work of \Cref{th:laminations_agree} is done with the following proposition.
\begin{proposition}\label{prop:limit_to_zero}
Let $(g_i )_{i\ge0}$ be a quasigeodesic sequence in $\Gamma$ converging to $z\in \partial \Gamma$. Then there is a $T_z \in \cvbar$ such that $z \mapsto \partial \pi_1(T_z)$ under $\partial \Gamma\to\partial\F$ with the property that for every $1\ne h\in\free$ we have
\[ \lim_{i \to \infty} \ell_{T_z}(g_ih) = 0.\]

\end{proposition}

\begin{proof}
Apply \Cref{lem:folding_rays} with $\lambda = 2$ to obtain the unit speed folding ray $\gamma\colon \I \to \X$ and denote by $G_t = e^{-t}\cdot \gamma(t)$ the rescaled folding path for which the associated folding maps are isometric on edges. Let $T_z$ be the limit of $G_t$ in $\cvbar$. By \Cref{lem:folding_rays}, the image of $z$ under the extension of the orbit map $\partial \Gamma \to \partial \F$ is $\partial \pi(T_z)$.

As in \Cref{lem:eventual_growth}, let $t_{g_ih}$ denote a time for which $g_i h$ has its length minimized along $\gamma(t)$. Also, for each $i \ge 0$ let $t_i$ denote a time for which the symmetric distance between $\gamma(t)$ and $g_iR$ is less than $K$. (Such a time exists by \Cref{lem:folding_rays}.) Hence, by definition of the symmetric distance on $\X$ we have 
\begin{align} \label{eq:lengths_close}
& e^{-K} \le \frac{\ell(\alpha|\gamma(t_i))}{\ell(\alpha|g_iR)} \le e^K,
\end{align}
for each conjugacy class $\alpha$. We will need the following claim:

\begin{claim*}\label{claim:times_agree}
There is a constant $B\ge0$ which is independent  of $i\ge0$ so that 
\[
\abs{t_i - t_{g_ih}} \le B.
\]
\end{claim*}
\begin{proof}[Proof of claim]
We will show that $B$ can be taken to be 
\[
\max \left \{2M+M\log_2 \frac{e^{M}e^{K} \ell(h\vert R)}{\epsilon},\; \log \left (\frac{c\cdot \ell(h|R)}{\epsilon} \right) +K \right \},
\]
where the constants $M,K,c$ are as in \Cref{lem:folding_rays}. To see this, first suppose that $t_i < t_{g_ih}$. Let $D = \floor{\frac{t_{g_ih}-t_i}{M}}$  so that $DM \le t_{g_ih}-t_i\le DM+M$ and consequently
\begin{equation}\label{eq:adjust_int_point}
\ell(g_ih\vert \gamma(t_{g_i h}-DM)) \le e^M \ell(g_i h\vert \gamma(t_i))
\end{equation}
since $\gamma$ is a directed geodesic. As in \Cref{lem:folding_rays} let $\epsilon$ be the length of the shortest loop appearing along the folding path $\gamma$. 
Then by definition of $t_{g_ih}$ we have
\[\ell(g_i h\vert \gamma(t_{g_ih} - M)) \ge \ell(g_i h\vert \gamma(t_{g_i h}))\ge \epsilon.\]
Applying the $(2,M)$--flaring condition inductively at times $t_{g_ih}-M, t_{g_ih}-2M, \ldots, t_{g_ih}-DM $, we find that
\[\ell(g_ih\vert \gamma(t_{g_ih}-DM)) \ge 2^{D-1}\ell(g_ih\vert\gamma(t_{g_ih}-M)) \ge 2^{D-1}\epsilon.\]
Combining with \Cref{eq:adjust_int_point} and rearranging gives
\[\frac{t_{g_ih}-t_i}{M} -2 \le D-1 \le \log_2 \frac{e^M \ell(g_i h\vert \gamma(t_i))}{\epsilon}.\]
Applying \Cref{eq:lengths_close} and isolating $t_{g_ih}-t_i$ now gives the desired bound 
\[t_{g_ih}-t_i  \le 2M + M\log_2 \frac{e^Me^K\ell(g_ih\vert g_i R))}{\epsilon} = 2M+M\log_2 \frac{e^{M}e^{K} \ell(h\vert R)}{\epsilon}.\]

Now suppose that $t_{g_ih} \le t_i$. Applying \Cref{lem:eventual_growth} to the conjugacy class $\alpha = g_i h$ and using \Cref{eq:lengths_close} then yields
\begin{align*}
e^K &\ge \frac{\ell(g_ih|\gamma(t_i))}{\ell(g_ih|g_iR)}\\
&\ge \frac{\ell(g_ih|\gamma(t_{g_ih}))}{c \cdot \ell(h|R)} e^{(t_i - t_{g_ih})}\\
&\ge  \frac{\epsilon}{c \cdot \ell(h|R)} e^{(t_i - t_{g_ih})}.
\end{align*}
One final rearrangement then gives the claimed bound
\begin{align*}
t_i - t_{g_ih} &\le K + \log \frac{c\cdot \ell(h|R)}{\epsilon} \qedhere
\end{align*}
\end{proof}

We next observe that $t_i \to \infty$ as $i\to \infty$. To see this, recall that the orbit map $g\mapsto gR$ gives a quasi-isometric embedding $\Gamma\to\os$ (see, e.g., \cite[Lemma 6.4]{DT1}). Since $(g_i)_{i\ge 0}$ is a geodesic in $\Gamma$, this implies $d_\os(g_0R,g_iR)\to \infty$ as $i\to\infty$. Therefore
\begin{equation}\label{eq:to_infinity}
t_i = d_{\X}(\gamma(0),\gamma(t_i)) \ge d_{\X}(g_0R,g_iR) - 2K \to\infty
\end{equation}
as $i\to\infty$, as claimed.

Finally, we can now compute
\begin{align*}
\lim_{i\to \infty}\ell_T(g_ih) &= \lim_{i \to \infty} \lim_{t \to \infty} \ell(g_ih\vert G_t) & \\
&= \lim_{i\to \infty} \lim_{t \to \infty} e^{-t} \cdot \ell(g_ih\vert \gamma(t)) &   \\
&\le \lim_{i\to \infty} \lim_{t\to\infty}e^{-t} (e^{(t-t_{g_ih})}\ell(g_ih\vert \gamma(t_{g_ih})) & (\Cref{lem:eventual_growth}) \\
&= \lim_{i \to \infty} e^{-t_{g_ih}}\cdot \ell(g_ih\vert \gamma(t_{g_ih})) &  \\
&\le \lim_{i\to \infty} e^{2B}e^{-t_i} \cdot \ell(g_ih\vert \gamma(t_i)) & (Claim~\ref{claim:times_agree})  \\
&= \lim_{i\to \infty}e^{2B+K}e^{-t_i} \cdot \ell(g_ih\vert g_iR) & (\Cref{eq:lengths_close})   \\
&=e^{2B+K}\ell(h|R) \cdot  \lim_{i\to \infty} e^{-t_i}  &  \\
&= 0 & (\Cref{eq:to_infinity}). 
\end{align*}
This completes the proof of the proposition.
\end{proof}

\begin{remark}[Unique ergodicity of $T_z$]
Although we will not need this fact, we note that Namazi--Pettet--Reynolds have recently shown that under the assumption that the orbit map $\Gamma \to \F$ is a quasi-isometric embedding, each equivalence class of trees appearing the image of $\partial \Gamma$ in $\partial \F$ is uniquely ergodic \cite{NPR}. (See \cite{NPR} and the references therein for a discussion of the various notions of unique ergodicity.) In particular, in the statement of \Cref{prop:limit_to_zero} (and therefore \Cref{th:laminations_agree}), the tree $T_z$ is unique up to rescaling.
\end{remark}

We can now conclude the proof of \Cref{th:laminations_agree}.
\begin{proof}[Proof of \Cref{th:laminations_agree}]
Choose a free basis $X$ of $\free$.
Let $z\in \partial \Gamma$ and let $(g_n)_{n=1}^\infty$ be a geodesic ray from $1$ to $z$ in the Cayley graph of $\Gamma$.

Let $(p,q)\in \Lambda_z$ be an arbitrary leaf of  $\Lambda_z$.  Then there is $1\ne h\in\free$ such that $(p,q)\in \Lambda_{z,h}$. For every $n\ge 1$ let $w_n$ be the cyclically reduced form of $g_n(h)$ over $X^{\pm 1}$. 
Let $\gamma$ be the bi-infinite geodesic from $p$ to $q$ in $\ca(\free,X)$. Let $v$ be the label of some finite subsegment of $\gamma$.  By \Cref{rem:Lambda_z}, there exists an infinite sequence $n_i\to\infty$ such that for all $i\ge 1$ $v$ is a subword of a cyclic permutation of $w_{n_i}^{\pm 1}$. By \Cref{prop:limit_to_zero}, we see that $\lim_{i \to \infty} \ell_{T_z}(w_{n_i}) = 0$.  Thus for every $\epsilon>0$ there exists a cyclically reduced word $w$ over $X^{\pm 1}$ with $\ell_{T_z}(w)\le \epsilon$ such that $v$ is a subword of $w$. Since $v$ was the label of an arbitrary subsegment of the geodesic from $p$ to $q$, by \Cref{rem:L(T)} it follows that $(p,q)\in L(T_z)$.  As $(p,q)\in \Lambda_z$ was arbitrary, we conclude that $\Lambda_z \subset L(T_z)$.

Since the orbit map $\Gamma\to\F$ is a quasi-isometric embedding, this orbit map extends to a $\Gamma$--equivariant injective continuous map $\partial \Gamma\to\partial\F$. Thus for any distinct $z_1,z_2\in\partial \Gamma$ we have $T_{z_1}\not\approx T_{z_2}$ and therefore, by \Cref{prop:CHR}, there do not exist $p,q,q'\in \partial \free$ such that $(p,q)\in L(T_{z_1})$ and $(p,q')\in L(T_{z_2})$. Since  $\Lambda=\cup_{z\in\partial \Gamma}\Lambda_z$ is diagonally closed by \Cref{th:Mj_lam}, it now follows that for every $z\in \partial \Gamma$ the lamination $\Lambda_z$ is diagonally closed. 

Let $z\in \partial \Gamma$ be arbitrary. \Cref{prop:keyQ} implies that $L(T_z)$ is diagonally closed. Since $T_z$ is free and arational, \cite[Theorem A]{CHR11} (see also \cite[Proposition~4.2]{bestvina2012boundary}) implies that $L(T_z)$ possesses a unique minimal sublimation and that $L(T_z)$ is obtained from this minimal sublamination by adding diagonal leaves. Therefore the only diagonally closed sublamination of $L(T_z)$ is $L(T_z)$ itself. We have already established that $\Lambda_z \subseteq L(T_z)$. Since $\Lambda_z$ is an algebraic lamination on $\free$ (see \Cref{rem:Lambda_z}) and since $\Lambda_z$ is diagonally closed, it follows that $\Lambda_z = L(T_z)$, as required.
\end{proof}

\section{Fibers of the Cannon--Thurston map}

Recall (c.f. \Cref{conv:main}) that we have fixed a convex cocompact subgroup $\Gamma\le\Out(\free)$ for which the extension group $E_\Gamma$ is hyperbolic. The short exact sequence $1\to \free\to E_\Gamma\to\Gamma\to 1$ thus gives rise to a surjective Cannon--Thurston map denoted $\ct\colon \partial\free\to\partial E_\Gamma$. We write $\deg(y)$ for the cardinality $\#\left((\ct)^{-1}(y)\right)$ of the fiber over $y\in \partial E_\Gamma$ and call this the \define{degree} of $y$.
In this section we use \Cref{th:laminations_agree} to describe the fibers of the Cannon--Thurston map. The key technical observation is the following.

\begin{lemma}
\label{lem:fiber_description}
Suppose $y\in \partial E_\Gamma$ has $\deg(y)\ge 2$. Then there is a unique point $z\in \partial \Gamma$ and a point $c\in T_z$ so that $(\ct)^{-1}(y) = \mathcal{Q}_{T_z}^{-1}(c)$. Moreover, for any $p\in (\ct)^{-1}(y)$ we have that $p$ is proximal for $L(T_z)$ and that 
\[(\ct)^{-1}(y) =\mathcal{Q}_{T_z}^{-1}(c) = \{p\}\cup\{q\in \partial \free \mid (p,q)\in L(T_z)\}.\]
\end{lemma}
\begin{proof}
Recall that since the orbit map $\Gamma\to \mathcal F$ is a quasi-isometric embedding and since $\Gamma$ and $\mathcal F$ are Gromov-hyperbolic, we have a $\Gamma$--equivariant topological embedding $\kappa\colon \partial \Gamma\to\partial F$. Thus to every $z\in\partial \Gamma$ we have an associated point $\kappa(z)\in \partial \mathcal F$ which is represented by an equivalence class of an arational tree $T_z\in \cvbar$. Moreover, by \Cref{th:laminations_agree}, for every $z\in  \partial \Gamma$ the action of $\free$ on $T_z$ is free and $\Lambda_z=L(T_z)$.
Note that since $\kappa$ is injective, \Cref{prop:CHR} implies that for every $p\in \partial \free$ there is at most one point $z\in\partial\Gamma$ for which $L(T_z)$ contains a leaf of the form $(p,q)$.

Suppose now that $\deg(y)=m\ge 2$, so that $(\ct)^{-1}(y)=\{p_1,\dots, p_m\}\subseteq \partial \free$ consists of $m\ge 2$ distinct points.
By \Cref{th:Mj_lam} and \Cref{th:laminations_agree}, we find that for each pair $1\le i < j\le m$ there is some $z_{ij}\in \partial \Gamma$ so that $(p_i,p_j) \in \Lambda_{z_{ij}} = L(T_{z_{ij}})$. The above observation (regarding \Cref{prop:CHR} and the injectivity of $\kappa$) shows there is in fact a unique such $z\in \partial\Gamma$; hence we have $(p_i,p_j)\in L(T_z)$ for all $1\le i<j\le m$.
This proves that for each $1\le i \le m$ the fiber $(\ct)^{-1}(y)$ has the claimed form
\[(\ct)^{-1}(y) = \{p_i\} \cup \{q\in \partial \free \mid (p_i,q)\in L(T_z)\}.\]
Moreover, since $p_1,\dotsc,p_m$ are all endpoints of leafs of $L(T_z)$, it is now immediate from \Cref{defn:proximal} that each point of $(\ct)^{-1}(y)$ is proximal for $L(T_z)$. Finally,  by \Cref{prop:keyQ} there is a point $c\in \hat T_z$ such that  $\mathcal{Q}_{T_z}^{-1}(c)=\{p_1,\dots, p_m\} = (\ct)^{-1}(y)$.
This concludes the proof of the lemma.
\end{proof}

\begin{definition}[$\Gamma$--essential points]\label{defn:essential}
A point $y\in \partial E_\Gamma$ is said to be \emph{$\Gamma$--essential} if there exists $z\in \partial \Gamma$ such that $\ct(x)=y$ for some $x\in\partial\free$ proximal for $L(T_z)$ (see \Cref{defn:proximal}). In this case, there is a unique such point $z\in \partial\Gamma$, which we denote $\zeta(y)\colonequals z$. 
\end{definition}

\Cref{prop:prox,prop:CHR} show that a point $x\in \partial\free$ can be proximal for $L(T_z)$ for at most one point $z\in \partial\Gamma$. Thus $\zeta(y)$ is clearly uniquely determined for any $\Gamma$--essential point with $\deg(y) =1$. This together with \Cref{lem:fiber_description} shows that \Cref{defn:essential} is justified in asserting that $\zeta(y)$ is uniquely determined. We also note that every $y\in \partial E_\Gamma$ with $\deg(y)\ge 2$ is $\Gamma$--essential by \Cref{lem:fiber_description}.

\subsection{Bounding the size of fibers of the Cannon--Thurston map}

Using \Cref{lem:fiber_description}, the $\mathcal{Q}$--index theory for very small $\R$--trees now easily gives a uniform bound on the cardinality of any fiber of the Cannon--Thurston map $\ct\colon\partial\free\to\partial E_\Gamma$.

\begin{theorem}\label{thm:main}\label{th:main_1}
Let $\Gamma\le \Out(\free)$ be purely atoroidal and convex cocompact, where $\free$ is a free group of finite rank at least $3$, and let $\ct\colon\partial \free\to\partial E_\Gamma$ denote the Cannon--Thurston map for the hyperbolic $\free$--extension $E_\Gamma$. Then for every $y\in \partial E_\Gamma$, the degree $\deg(y) = \#\left((\ct)^{-1}(y)\right)$ of the fiber over $y$ satisfies
\[1 \le \deg(y) \le 2\rank(\free).\]
In  particular, the fibers Cannon--Thurston map are all finite and of uniformly bounded size.
\end{theorem}

\begin{proof}
Fix any point $y\in \partial E_\Gamma$. Since the Cannon--Thurston map is surjective, we clearly have $(\ct)^{-1}(y)\neq \varnothing$. Thus $\deg(y)\ge 1$. If $\deg(y)=1$ there is nothing to prove, so assume $\deg(y) = m\ge 2$. By \Cref{lem:fiber_description}, there exists a free arational tree $T_z\in\cvbar$ and a point $c\in\hat{T_z}$ so that $(\ct)^{-1}(y) = \mathcal{Q}_{T_z}^{-1}(c)$.
\Cref{th:index-bound} then gives
\[m-2 = \ind(c) \le \ind(T_z)\le 2\rank(\free)-2.\qedhere\] 
\end{proof}

\subsection{Rational points and the Cannon--Thurston map}

A point in the boundary $\partial G$ of a word-hyperbolic group $G$ is said to be \emph{rational} if it is equal to the limit $g^{\infty} \colonequals \lim_{n\to\infty} g^n$ in $G\cup \partial G$ for some infinite order element $g\in G$. A point in $\partial G$ is \emph{irrational} if it is not rational. Our next result analyzes the fibers of the Cannon--Thurston map $\ct$ over rational points  of $\partial \free$ and $\partial E_\Gamma$.

\begin{theorem} \label{th:rational_fibers}
Suppose that $1 \to \free \to E_\Gamma \to \Gamma \to 1$ is a hyperbolic extension with $\Gamma \le \Out(\free)$ convex cocompact. 
Consider a rational point $g^\infty\in \partial E_\Gamma$, where $g\in E_\Gamma$ has infinite order.
\begin{enumerate}
\item Suppose that $g^k$ is equal to $w\in \free\lhd E_\Gamma$ for some $k\ge 1$ (i.e., $g$ projects to a finite order element of $\Gamma$). Then $(\ct)^{-1}(g^\infty) = \{w^\infty\}\subset \partial \free$ and so $\deg(g^\infty) = 1$.
\item Suppose that $g$ projects to an infinite-order element $\phi\in\Gamma$. Then there exists $k\ge 1$ such that the automorphism $\Psi\in\Aut(\free)$ given by $\Psi(w) = g^kwg^{-k}$ is forward rotationless (in the sense of \cite{FH11,CoulboisHilion-botany}) and its set $\mathrm{att}(\Psi)$ of attracting fixed points in $\partial \free$ is exactly $\mathrm{att}(\Psi)=(\ct)\inv(g^\infty)$. Moreover, $g^\infty$ is $\Gamma$--essential and $\zeta(g^\infty) = \phi^\infty$.
\end{enumerate}
\end{theorem}
\begin{proof}
First suppose $g^k = w\in \free$ for some $k\ge 1$. By continuity and $\free$--equivariance of the Cannon--Thurston map, it is immediate that $\ct$ sends $w^\infty \in \partial \free$ to $g^\infty \in \partial E_\Gamma$ (note that $(g^k)^\infty = g^\infty$ in $\partial E_\Gamma$). Thus $\{w^\infty\}\subset (\ct)^{-1}(g^\infty)$. Since for every $z\in\partial\Gamma$, $T_z$ is a free arational tree, there do not exist $z\in \partial \Gamma$ and $p\in \partial \free$ such that $(p,w^\infty)\in L(T_z)$. Therefore  \Cref{th:Mj_lam} and \Cref{th:laminations_agree} imply that $(\ct)^{-1}(g^\infty) \subset\{w^\infty\}$ and part (1) is verified. 

Now suppose that $g$ projects to an infinite order element $\phi\in \Gamma$.
As explained in \Cref{sect:fully-irred}, we may choose $k\ge 1$ so that the automorphism $\Psi\in \Aut(\free)$ given by $w\mapsto g^k w g^{-k}$ is forward rotationless. Let $p\in att(\Psi)$ be a locally attracting fixed point for the left action of $\Psi$ on $\partial \free$.  By \Cref{cor:attr}, $p$ is a proximal point for $L(T_\phi)$, and \Cref{cor:phi} moreover shows that $T_\phi = T_z$ for the point $z\colonequals \phi^\infty\in \partial\Gamma$.

 Recall that by \Cref{prop:action} the map $\ct\colon\partial \free\to\partial E_\Gamma$ is $E_\Gamma$--equivariant, and that $g^k$ acts on $\partial \free$ by $\Psi$, that is  $g^kq=\Psi(q)$ for every $q\in\partial \free$. 
Since $p=g^kp$ is a local attractor for $\Psi$, the $E_\Gamma$--equivariance and continuity of $\ct$ ensure that $\ct(p)=g^k\ct(p)$ is a local attractor for the action of $g^k$ on $\partial E_\Gamma$. Since $g$ is an element of infinite order in a word-hyperbolic group $E_\Gamma$, $g^\infty$ is the unique local attractor for the action of $g$ on $\partial E_\Gamma$. Thus we may conclude that in fact $\ct(p)=g^\infty$. This proves that $att(\Psi)\subseteq (\ct)^{-1}(g^\infty)$. Since $p$ is proximal for $L(T_{z})$, we also see that $g^\infty$ is $\Gamma$--essential and that $\zeta(g^\infty)=z = \phi^\infty$, as claimed.

Now, if $\deg(g^\infty)=1$, it follows that $\#att(\Psi)\le 1$ so that $att(\Psi) = \{p\}= (\ct)^{-1}(g^\infty)$, as required. On the other hand, if $\deg(g^\infty)\ge 2$, then  \Cref{lem:fiber_description} provides a point $c\in T_\phi$ so that $att(\Psi)\subset (\ct)^{-1}(g^\infty) = \mathcal{Q}_{T_\phi}^{-1}(c)$. Part (2) of \Cref{prop:att} then ensures that $(\ct)^{-1}(g^\infty)=att(\Psi)$. Thus claim (2) is verified.
\end{proof}

Kapovich and Lustig showed~\cite{KapLusCT} that for the Cannon--Thurston map $\partial\free\to\partial E_{\langle\phi\rangle}$ associated to a cyclic group generated by an atoroidal fully irreducible $\phi\in \Out(\free)$, every point $y\in \partial E_{\langle\phi\rangle}$ with $\deg(y)\ge 3$ is rational and has the form $y=g^\infty$ for some $g\in E_{\langle\phi\rangle}-\free$. 
(Note that here $\partial\langle\phi\rangle$ consists of two points $\phi^{\pm\infty}$, both of which are rational.) 
Here we show that this result need not hold in the general setting of convex cocompact subgroups. Rather, we find that the Cannon--Thurston map, via the assignment $y\mapsto \zeta(y)$ for $\Gamma$--essential points, detects the following relationship between rationality/irrationality in $\partial E_\Gamma$ and $\partial \Gamma$.

\begin{theorem}\label{th:rational_and_irrational}
Suppose that $1 \to \free \to E_\Gamma \to \Gamma \to 1$ is a hyperbolic extension with $\Gamma \le \Out(\free)$ convex cocompact. Then the following hold:
\begin{enumerate}
\item If $y\in \partial E_\Gamma$ has $\deg(y)\ge 3$ and $\zeta(y)\in\partial \Gamma$ is rational, then $y$ is rational.
\item If $y\in \partial E_\Gamma$ has $\deg(y)\ge 2$ and $\zeta(y)\in\partial\Gamma$ is irrational, then $y$ is irrational. 
\end{enumerate}
\end{theorem}
Our argument for part (1) of the above theorem is similar to the proof of part (3) of Theorem~5.5 in \cite{KapLusCT}. The proof is included here for completeness. 
\begin{proof}
Suppose first that $y\in \partial E_\Gamma$ is such that $\deg(y)\ge 3$ and $\zeta(y)\in\partial \Gamma$ is rational. Thus $\zeta(y)=\phi^\infty$ for some atoroidal fully irreducible $\phi\in\Gamma$. 
Then $(\ct)^{-1}(y)=\mathcal Q_{T_\phi}^{-1}(x)$ for some $x\in \overline{T_\phi}$. Let $k\ge 1$ be such that $\psi=\phi^k$ is FR.  Note that $T_\phi=T_\psi$.  Choose some homothety $H$ of $\overline T_\phi$.  \Cref{eqn:homothety-action} in \Cref{sect:fully-irred} implies that $H$ acts on $\free$--orbits of points of $\overline T_\phi$.  Since there are only finitely many $\free$--orbits of points of $\overline T_\phi$ with $\mathcal Q_{T_\phi}$--preimage of cardinality $\ge 3$ (\Cref{prop:Qmap}), some positive power of $H$ preserves every such orbit and, in particular, preserves the orbit of $x$. Thus, after replacing $k$ by $kt$ for some integer $t\ge 1$, we may assume that $Hx=wx$ for some $w\in\free $. Then $w^{-1}Hx=x$. The homothety $H_1=w^{-1}H$ represents some $\Psi\in\Aut(\free)$ whose outer automorphism class is $\psi$.  Moreover, $x$ is the center of the homothety $H_1$. Therefore by  \Cref{prop:att}  it follows that $\mathcal Q_{T_\phi}^{-1}(x)=att(\Psi)$. Thus $(\ct)^{-1}(y)=att(\Psi)$. 

Choose $g\in E_\Gamma$ so that the automorphism $h\mapsto ghg^{-1}$ of $\free$ is exactly $\Psi$. Note that $g$ projects to $\psi\in\Gamma$ and thus $g$ has infinite order in $E_\Gamma$. Let $p\in att(\Psi)$. We have $p=\Psi(p)=gp$ and $\ct(p)=y$. By $E_\Gamma$--equivariance of $\ct$ it follows that $gy=y$. Since $g$ is an element of infinite order in a word-hyperbolic group $E_\Gamma$, it follows that $y=g^{\pm\infty}$ is rational. This proves claim (1).

Next suppose that $\deg(y)\ge 2$ and that $\zeta(y)$ is irrational. Assuming, on the contrary, that $y$ is rational, we have $y=g^\infty$ for some non-torsion element $g\in E_\Gamma$. Since $\deg(y)\ge 2$, \Cref{th:rational_fibers} (1) implies that $g$ projects to an element of infinite order $\phi$ of $\Gamma$. But then $\zeta(y) = \phi^\infty$ is rational by \Cref{th:rational_fibers} (2), contradicting the assumption that $\zeta(y)$ was irrational. Therefore $y$ is indeed rational and claim (2) holds.
\end{proof}

\subsection{Conical limit points}  
Recall that every non-elementary subgroup of a word-hyperbolic group $G$ acts as a convergence group on the Gromov boundary $\partial G$. 
If a group $H$ acts as a convergence group on a compact metrizable space $Z$, a point $z\in Z$ is called a \emph{conical limit point} for the action of $H$ on $Z$ if there exist an infinite sequence $h_n$ of distinct elements of $H$ and a pair of distinct points $z_-,z_+\in Z$ such that $\lim_{n\to\infty} h_n z = z_+$ and that $(h_n|_{Z\setminus \{z\}} )_n$ converges uniformly on compact subsets to the constant map $c_{z_-} \colon Z\setminus \{z\}\to Z$ sending $Z\setminus \{z\}$ to $z_-$. It is also known that if $H\le G$ is a non-elementary subgroup of a word-hyperbolic group $G$, then $z\in \partial G$ is a conical limit point for the action of $H$ on $\partial G$ if and only if there exists an infinite sequence of distinct elements $h_n\in H$ such that all $h_n$ lie in a bounded Hausdorff neighborhood of a geodesic ray from $1$ to $z$ in the Cayley graph of $G$.  We refer the reader to \cite{JKLO} for more details and background regarding conical limit points. 

\begin{theorem}\label{th:conical}\label{cor:main_4}
Let $\Gamma\le\Out(\free)$ be purely atoroidal and convex cocompact. If $y\in \partial E_\Gamma$ is $\Gamma$--essential, then $y$ is not a conical limit point for the action of $\free$ on $\partial E_\Gamma$. 
In particular, if $\deg(y)\ge 2$ or if $y=g^\infty$ for some $g\in E_\Gamma$ projecting to an infinite-order element of $\Gamma$, then $y$ is not a conical limit point for the action of $\free$.
\end{theorem}

\begin{proof}
Choose a free basis $X$ of $\free$. Suppose that $y\in \partial E_\Gamma$ is $\Gamma$--essential, and let $z=\zeta(y)\in\partial \Gamma$ so that $y=\ct(p)$ for some $p\in\partial\free$ proximal to $L(T_z)=\Lambda_z$.
Then every freely reduced word over $X^{\pm 1}$ which occurs infinitely many times in the geodesic ray from $1$ to $p$ in $\ca(\free,X)$ is a leaf-segment for $L(T_z)=\Lambda_z$.
Therefore \cite[Theorem~B]{JKLO} implies that $y=\ct(p)$ is not a conical limit point for the action of $\free$ on $\partial E_\Gamma$, as claimed. The remaining assertions now follow from \Cref{lem:fiber_description} and \Cref{th:rational_fibers}.
\end{proof}

\section{Discontinuity of $z\in \partial \Gamma \mapsto \Lambda_z \in \L(\free)$}\label{sec:continuity}

In this section, we answer a question of Mahan Mitra using our hyperbolic extension $E_\Gamma$. As stipulated in \Cref{conv:main}, our fixed finitely generated subgroup $\Gamma \le \Out(\free)$ has a quasi-isometric orbit map into the free factor complex and gives rise to an exact sequence of word-hyperbolic groups
\[1 \longrightarrow \free \longrightarrow E_\Gamma \longrightarrow \Gamma \longrightarrow 1.\]
Thus each point $z \in \partial \Gamma$ has an associated ending lamination $\Lambda_z$ (\Cref{def:Mitra}) and we consider the map $F\colon \partial \Gamma \to \mathcal{L}(\free)$ defined by $F(z)=\Lambda_z$. 
Here $\L(\free)$ is the set of laminations equipped with the Chabauty topology (\Cref{def:Chabauty}). In his work on Cannon--Thurston maps for normal subgroups of hyperbolic groups, Mitra asked whether this map  $F$ is continuous.  We answer this question by producing an explicit example for which $F\colon \partial \Gamma \to \mathcal{L}(\free)$ is not continuous. This is done in \Cref{ex:disc}.

Before turning to this example, we establish a ``subconvergence'' property for the map $F\colon \partial \Gamma \to \mathcal{L}(\free)$. This is the strongest positive result that one can give about continuity with respect to the Chabauty topology on $\L(\free)$. Recall that for a lamination $L \in \L(\free)$ the notation $L'$ denotes the set of accumulation points of $L$, in the usual topological sense.

\begin{proposition}\label{th:lamination_continuity}
Let $\Gamma\le\Out(\free)$ be purely atoroidal and convex cocompact, 
and let $\Lambda_z\in\L(\free)$ denote the ending lamination associated to $z\in \partial \Gamma$. Then for any sequence $z_i$ in $\partial \Gamma$ converging to $z$ and any subsequence limit $L$ of the corresponding sequence $\Lambda_{z_i}$ in $\L(\free)$, we have
\[
\Lambda'_z \subset L \subset \Lambda_z.
\]
\end{proposition}

Before proving \Cref{th:lamination_continuity}, we first recall a characterization of convergence in the Chabauty topology. Let $X$ be a locally compact metric space and let $C(X)$ be the space of closed subsets of $X$ equipped with the Chabauty topology. Recall that $C(X)$ is compact. The following lemma is well known; see \cite{canary2006fundamentals}.

\begin{lemma}[Chabauty convergence] \label{lem:Cabauty_convergence}
For a locally compact metric space $X$, a sequence $C_i$ converges to $C$ in $C(X)$ if and only if the following conditions are satisfied:
\begin{enumerate}
\item For each $x_{i_k} \in C_{i_k}$, whenever $x_{i_k} \to x$ in $X$ as $k \to \infty$ it follows that $x\in C$.
\item For each $x\in C$, there is are $x_i \in C_i$ with $x_i \to x$ in $X$ as $i \to \infty$.
\end{enumerate}
\end{lemma}

As ``weak evidence'' towards continuity of the map $F$, Mitra proves the following proposition \cite{MitraEndingLams}, which essentially amounts to verifying $(1)$ in \Cref{lem:Cabauty_convergence}.

\begin{proposition}[Proposition $5.3$ of \cite{MitraEndingLams}]\label{prop:subconvergence}
If $z_i \to z$ in $\partial \Gamma$ and if $(p_i,q_i)\in\Lambda_{z_i}$ converge to $(p,q)$ in $\DB$, then $(p,q)\in\Lambda_z$.
\end{proposition}

\begin{remark}[Semi-continuity of the map $T \mapsto L(T)$]
In \cite{CHL2}, the authors remark that if $(T_i)_{i\ge 0}$ is a sequence of trees in $\cvbar$ converging to a tree $T$, then any subsequence limit $L$ of the corresponding sequence of dual laminations $(L(T_i))_{i\ge 0}$ in $\L(\free)$ is contained in $L(T)$. They elaborate this statement in \cite[Proposition~1.1]{CHLunpub}.  Combining this general fact with \Cref{th:laminations_agree} gives an alternative proof of \Cref{prop:subconvergence}.
\end{remark}

We now turn to the proof of \Cref{th:lamination_continuity}.

\begin{proof}[Proof of \Cref{th:lamination_continuity}]
Suppose that $z_i \to z$ in $\partial \Gamma$. Then by \Cref{prop:subconvergence} and \Cref{lem:Cabauty_convergence}, if $L$ is any subsequential limit of $\Lambda_{z_i}$ in the Chabauty topology it follows that $L \subset \Lambda_z$. 
By \Cref{th:laminations_agree}, $\Lambda_z = L(T_z)$ for the free and arational tree $T_z$; we also know that  $L(T_z)$ is the diagonal closure of its unique minimal sublamination $L'(T_z)$ by \cite[Theorem~A]{CHR11}. In particular, we must have that $L \supset L'(T_z) = \Lambda'_z$. Hence, $ \Lambda'_z \subset L \subset \Lambda_z$, as required.
\end{proof}

We conclude this section by producing an example of a hyperbolic extensions $E_\Gamma$ for which $F \colon \partial \Gamma \to \L(\free)$ is not continuous. Before explaining this example, we briefly recall some facts related
to the index theory of free group automorphisms. We refer the reader
to \cite{CoulboisHilion-botany,kapovich2014invariant,CoulboisHilion-index,CHR11} for more details.

Let $\phi\in\Out(\free)$ be an atoroidal fully irreducible element and
let $h \colon G\to G$ be a train track representative of $\phi$ (thus $h$ is
necessarily expanding and irreducible). By replacing $h$ with a sufficiently large
positive power and possibly subdividing $G$, we may further assume the following:
that the endpoints of all INPs in
$G$ (if any are present) are vertices of $G$, that every periodic vertex of $G$ is
fixed by $h$, that every periodic direction in $G$ has period $1$ and
that every periodic INP (if any are present) in $G$ has period $1$.
Here, the acronym INP stands for irreducible Nielsen path; see \cite{bestvina1992train, bestvina1997laminations, FH11} for background.
In the discussion below,  if $(p,q)\in L$ for some algebraic lamination $L$ on $\free$, we will often refer to a geodesic $\mathfrak l$ from $p$ to $q$ in $\tilde G$ as a leaf of $L$.

Recall that the \emph{Bestvina-Feighn-Handel lamination}
$L_{BFH}(\phi)$ is an algebraic lamination on $\free$ consisting of
all $(p,q)\in \partial^2\free$ such that for every finite subpath $\tilde\gamma$
of the geodesic in $\tilde\Gamma$ connecting $p$ to $q$, the projection
$\gamma$ of $\tilde\gamma$ to $G$ is a subpath of $h^n(e)$ for some $n\ge 1$ and
some (oriented) edge $e$ of $G$.  If $v$ is a periodic vertex of $G$
and $e$ is an edge starting with $v$ defining a periodic direction at
$v$ (so that $h(e)$ starts with $e$), then $e$ determines a semi-infinite reduced edge-path $\rho_e$ in
$G$ called the \emph{eigenray} of $h$ corresponding to $v$. Namely,
$\rho_e$ is defined as the path such that for every $n\ge 1$, $h^n(e)$
is an initial segment of $\rho_e$. It is known \cite{kapovich2014invariant} that
$L_{BFH}(\phi)\subseteq L(T_\phi)$ is the unique minimal sublamination
of
$L(T_\phi)$; that is, $L_{BFH}(\phi)$ is the unique minimal (with respect
to inclusion) nonempty subset of $L(T_\phi)$ which is itself an
algebraic lamination on $\free$.  It is also known that $L(T_\phi)$ is
the ``transitive closure'' of $L_{BFH}(\phi)$; that is, $L(T_\phi)$ is
the smallest diagonally closed (in the sense defined in \Cref{sec:laminations_on_free}) subset of $\partial^2 \free$ which contains $L_{BFH}(\phi)$ and is itself a lamination. Moreover,
$L(T_\phi)\setminus L_{BFH}(\phi)$ consists of finitely many
$\free$--orbits of points of $\partial^2\free$ called \emph{diagonal
  leaves} of $L(T_\phi)$, and \cite{kapovich2014invariant} gives a precise description of these diagonal leaves in terms of the train track $h$: If $v$ is a periodic
vertex and $e,e'$ are distinct periodic edge of $G$ with origin $v$,
then any lift to $\tilde\Gamma$ of the biinfinite path
$\rho_e^{-1}\rho_{e'}$ is a leaf of $L(T_\phi)$; such a leaf is called a
\emph{special leaf}. Some of the special leaves already belong to
$L_{BFH}(\phi)$ (this happens precisely when the turn $e,e'$ is
``taken'' by $h$). Special leaves that do not belong to $L_{BFH}(\phi)$
are necessarily diagonal. If $h$ has no periodic INPs, then all
diagonal leaves of  $L(T_\phi)$ arise in this way; that is, every
diagonal leaf is special. If $h$ has some periodic INPs, then
$L(T_\phi)$ admits diagonal leaves of additional kind, but their
precise description is not needed here (see \cite{kapovich2014invariant} for details).

\begin{example}[Discontinuity of $F \colon \partial \Gamma \to \L(\free)$] \label{ex:disc}
We now construct an example of a purely atoroidal convex cocompact
subgroup $\Gamma\le \Out(\free)$, with $\rank(\free)=3$, such that
the map $\partial \Gamma\to\mathcal L(\free)$ given by $z\mapsto \Lambda_z$
is not continuous with respect to the Chabauty topology on $\L(\free)$.

Suppose that we are given automorphisms $\phi$ and $\psi$ of $\free = F(a,b,c)$, the free group of rank $3$, with the following properties:
\begin{enumerate}
\item $\phi$ and $\psi$ are atoroidal and fully irreducible.
\item $\phi$ and $\psi$ are positive with respect to the basis
  $\{a,b,c\}$. Thus we can represent $\phi$ and $\psi$ by  train track
  maps on the rose $R_3$ with a single vertex $v$ and petals
  corresponding to $a,b,c$;  we denote these train track maps by $f$
  and $g$ accordingly. We further assume that we have replaced $f$ and
  $g$ by appropriate positive powers so that for each of $f,g$  every periodic vertex of $R_3$ is
fixed, that every periodic direction in $R_3$ has period $1$, and
that every periodic INP in $R_3$ (if any are present) has period $1$.
\item For the map $f$  the directions corresponding to the three edges
  of $R_3$ labelled $a,b,c$ are periodic.
\item The map $g$ has $4$ periodic directions at $v$, given by the
  edges labelled by $a,c,a^{-1},b^{-1}$. Moreover, $g$ has no periodic
  INPs.
 \end{enumerate}
At the end of this example, we will give references for where one can find automorphisms satisfying these conditions.

By \cite{bestvina1997laminations, KLping, DT1},
we may replace $\phi$ and $\psi$ by further positive powers such that $\Gamma = \langle \phi, \psi \rangle$ is a purely atoroidal, convex cocompact subgroup of $\Out(\free)$. Hence, the corresponding extension $E_\Gamma$ is hyperbolic. We will show that the map  $F\colon\partial \Gamma \to \L(\free)$ defined by $F(z) = \Lambda_z$ is not continuous.

For a leaf $\mathfrak l$ of a lamination $L\in\L(\free)$, we say that $\mathfrak l$
is \emph{positive} (correspondingly \emph{negative}) if $\mathfrak l$ is
labelled by a positive (correspondingly negative) bi-infinite
word in $F(a,b,c)$, and we say that $\mathfrak l$ is \emph{mixed} if it is
neither positive nor negative.
Note that, since the automorphism $\psi$ is positive, the
definition of $L_{BFH}(\psi)$ implies that every leaf of $L_{BFH}(\psi)$ is, up to a
flip, labelled by a positive biinfinite word in $F(a,b,c)$. Also recall that, as discussed
above, $L_{BFH}(\psi)$ is the unique minimal sublamination of $L(T_\psi)$.
Since $g$ has no periodic INPs, every mixed leaf of $L(T_\psi)$ is
special. Hence, each mixed leaf of $L(T_\psi)$ is, up to the $\free$--action and the flip,
labelled by a word of
the form $\rho_1^{-1}\rho_2$ where $\rho_1,\rho_2$ are two eigenrays
of $g$ corresponding to two distinct periodic directions at $v$.

We establish following: $(i)$ there are exactly $2$ mixed diagonal
leaves of $L(T_\psi)$, up to the $\free$--action and the flip, $(ii)$
if $\phi^n L(T_\psi)$ converges to $L$ in $\L(\free)$ then $L$ contains
at most $2$ mixed leaves, up to the $\free$--action and the flip;  $(iii)$ the lamination
$L(T_\phi)$ contains at least $3$ distinct mixed diagonal leaves, up to the
$\free$--action and the flip.

To see $(i)$, note that since we assumed that $g$ has no periodic
INPs, mixed leaves of $L(T_\psi)$ are leaves of
$L(T_\psi) \setminus L_{BFH}(\psi)$ and hence are diagonal and special
for $L(T_\psi)$. Recall that $g$ has exactly four periodic directions at
$v$, namely $a,c,a^{-1},b^{-1}$. Thus $g$ has 4 eigenrays starting at
$v$: positive eigenrays $\rho_a$, $\rho_c$ and negative eigenrays
$\rho_{a^{-1}}$, $\rho_{b^{-1}}$. Up to a flip, every mixed leaf of
$L(T_\psi)$ is then labeled by either $(\rho_a)^{-1}\rho_c$ or $(\rho_{a^{-1}})^{-1} \rho_{b^{-1}}$.
Thus $(i)$ is verified.

A similar argument can be used to prove $(ii)$, although a bit of care
is needed here in using the definition of the Chabauty topology on
$\L(\free)$. Suppose $L\in \L(\free)$ is the limit of $\phi^n
L(T_\psi)$ as $n\to\infty$. Any mixed leaf $\mathfrak l$ of $L$, up to a flip, is
labelled by a word of the form $W^{-1}Z$, where $W$ and $Z$ are each
positive rays from the identity in $F(a,b,c)$ starting with distinct
symbols.  Since $\lim_{n\to\infty} \phi^n
L(T_\psi)=L$ in the Chabauty topology, and since $\phi$ is a positive
automorphism, every such mixed leaf $\mathfrak l$ of
$L$ must be a subsequential limit of mixed leaves of $\phi^n
L(T_\psi)$, which are labelled by words of the form
$\phi^n\left((\rho_a)^{-1}\rho_c\right)$ or $\phi^n\left((\rho_{a^{-1}})^{-1} \rho_{b^{-1}}
   \right)$. Finally note that the assumptions on
$f$ imply that if
$n_i,m_i\to\infty$ are two sequences of indices such that some mixed
leaves of $\phi^{n_i}
L(T_\psi)$ labelled by words of the form
$\phi^{n_i}\left((\rho_a)^{-1}\rho_c\right)$ converge to a mixed leaf
${\mathfrak l}_1$ of $L$ and that some leaves of $\phi^{m_i}
L(T_\psi)$ labelled by words of the form
$\phi^{m_i}\left((\rho_a)^{-1}\rho_c\right)$ converge to a mixed leaf
${\mathfrak l}_2$ of $L$, then the leaves ${\mathfrak l}_1,{\mathfrak l}_2$ have the same label
(up to a shift). The same holds when $(\rho_a)^{-1}\rho_c$ is replaced
by $(\rho_{a^{-1}})^{-1} \rho_{b^{-1}}$. It follows that, up to the $\free$--action and the
flip, there are at most 2 mixed leaves in $L$, and $(ii)$ is verified.

Finally, we observe $(iii)$. Let $r_a,r_b,r_c$
be the $f$--eigenrays in $R_3$ corresponding to the $f$--periodic directions
$a,b,c$ at $v$.  Thus $r_a,r_b,r_c$ are positive semi-infinite words. Then there exist special
mixed leaves in $L(T_\phi)$ labelled by $(r_a)^{-1}r_b$, $(r_a)^{-1}r_c$,
$(r_b)^{-1}r_c$. These leaves are distinct, up to the $\free$--action
and the flip. Thus $(iii)$ is verified.

Now let $\phi^\infty=\lim_{n\to\infty} \phi^n \in\partial\Gamma$ and
$\psi^\infty=\lim_{n\to\infty} \psi^n \in\partial\Gamma$. We know that
the orbit map $\Gamma\to\F$ induces an embedding $\partial
\Gamma\to\partial \F$ which takes $\phi^\infty$ to $T_\phi$ and
$\psi^\infty$ to $T_\psi$. We argue that $F\colon \partial \Gamma \to \L(\free)$ is not continuous by
contradiction.
Indeed suppose that $F$ is continuous. Then by \Cref{th:laminations_agree}
\begin{align*}
\lim_{n\to \infty} \phi^n L(T_\psi) &= 
 \lim_{n\to \infty} \phi^n \Lambda_{\psi^\infty} \\
& = \lim_{n\to \infty} \phi^n F(\psi^\infty) \\
& = \lim_{n\to \infty}  F(\phi^n \psi^\infty)\\
&=F(\phi^\infty)=\Lambda_{\phi^\infty}= L(T_\phi),
\end{align*}
where convergence in $\L(\free)$ is with respect to the Chabauty
topology. 
Together with $(ii)$, the fact that $\lim_{n\to\infty} \phi^n
L(T_\psi)=L(T_\phi)$ implies that, up to the $\free$--action and the flip,
there are at most 2 distinct mixed leaves in $L(T_\phi)$. However,
this contradicts $(iii)$. Thus $F$ is not continuous.

To complete the example, it only remains to give an example of
automorphisms $\phi$ and $\psi$ satisfying conditions (1)--(4). The
automorphism $\phi$ can be taken to be the automorphism $\alpha_3$
constructed by Ja\"ger and Lustig in \cite{jager2008free}. This automorphism is given by $f(a) =
abc$, $f(b) = bab$, and $f(c) = cabc$, and each of the required
properties is verified by Ja\"ger and Lustig. For the automorphism
$\psi$, we may take (a rotationless power of) the automorphism constructed by Pfaff 
in Example $3.2$ of \cite{pfaff2013out}. 
This is the automorphism $g(a) = cab$, $g(b) =ca$, and $g(c) =
acab$, and the required properties are established by Pfaff. This completes the example.
\end{example}

\bibliographystyle{alphanum}

\begin{thebibliography}{JKLO}

\bibitem[BF1]{BF-Outer}
Mladen Bestvina and Mark Feighn.
\newblock Outer limits.
\newblock Preprint 1994; http://andromeda.rutgers.edu/ \linebreak
  feighn/papers/outer.pdf, 2013.

\bibitem[BF2]{BF14}
Mladen Bestvina and Mark Feighn.
\newblock Hyperbolicity of the complex of free factors.
\newblock {\em Adv. Math.}, 256:104--155, 2014.

\bibitem[BFH]{bestvina1997laminations}
Mladen Bestvina, Mark Feighn, and Michael Handel.
\newblock Laminations, trees, and irreducible automorphisms of free groups.
\newblock {\em Geom. Funct. Anal.}, 7(2):215--244, 1997.

\bibitem[BH]{bestvina1992train}
Mladen Bestvina and Michael Handel.
\newblock Train tracks and automorphisms of free groups.
\newblock {\em Ann. of Math.}, pages 1--51, 1992.

\bibitem[Bow1]{Bow07}
Brian~H. Bowditch.
\newblock The {C}annon-{T}hurston map for punctured-surface groups.
\newblock {\em Math. Z.}, 255(1):35--76, 2007.

\bibitem[Bow2]{Bow13}
Brian~H. Bowditch.
\newblock Stacks of hyperbolic spaces and ends of 3-manifolds.
\newblock In {\em Geometry and topology down under}, volume 597 of {\em
  Contemp. Math.}, pages 65--138. Amer. Math. Soc., Providence, RI, 2013.

\bibitem[BR1]{BakerRiley}
Owen Baker and Timothy~R. Riley.
\newblock Cannon-{T}hurston maps do not always exist.
\newblock {\em Forum Math. Sigma}, 1:e3, 11, 2013.

\bibitem[BR2]{bestvina2012boundary}
Mladen Bestvina and Patrick Reynolds.
\newblock The boundary of the complex of free factors.
\newblock {\em Duke Math. J.}, 164(11):2213--2251, 2015.

\bibitem[Bri]{Brink}
Peter Brinkmann.
\newblock Hyperbolic automorphisms of free groups.
\newblock {\em Geom. Funct. Anal.}, 10(5):1071--1089, 2000.

\bibitem[CH1]{CoulboisHilion-botany}
Thierry Coulbois and Arnaud Hilion.
\newblock Botany of irreducible automorphisms of free groups.
\newblock {\em Pacific J. Math.}, 256(2):291--307, 2012.

\bibitem[CH2]{CoulboisHilion-index}
Thierry Coulbois and Arnaud Hilion.
\newblock Rips induction: index of the dual lamination of an {$\Bbb{R}$}-tree.
\newblock {\em Groups Geom. Dyn.}, 8(1):97--134, 2014.

\bibitem[CHL1]{CHLunpub}
Thierry Coulbois, Arnaud Hilion, and Martin Lustig.
\newblock Which {$\mathbb R$}-trees can be mapped continuously to a current?
\newblock Draft preprint, 2006.

\bibitem[CHL2]{CHL0}
Thierry Coulbois, Arnaud Hilion, and Martin Lustig.
\newblock Non-unique ergodicity, observers' topology and the dual algebraic
  lamination for {$\Bbb R$}-trees.
\newblock {\em Illinois J. Math.}, 51(3):897--911, 2007.

\bibitem[CHL3]{CHL1}
Thierry Coulbois, Arnaud Hilion, and Martin Lustig.
\newblock {$\mathbb R$}-trees and laminations for free groups. {I}. {A}lgebraic
  laminations.
\newblock {\em J. Lond. Math. Soc. (2)}, 78(3):723--736, 2008.

\bibitem[CHL4]{CHL2}
Thierry Coulbois, Arnaud Hilion, and Martin Lustig.
\newblock {$\mathbb R$}-trees and laminations for free groups. {II}. {T}he dual
  lamination of an {$\mathbb R$}-tree.
\newblock {\em J. Lond. Math. Soc. (2)}, 78(3):737--754, 2008.

\bibitem[CHR]{CHR11}
Thierry Coulbois, Arnaud Hilion, and Patrick Reynolds.
\newblock Indecomposable {$F_N$}-trees and minimal laminations.
\newblock {\em Groups Geom. Dyn.}, 9(2):567--597, 2015.

\bibitem[CL]{CL95}
Marshall~M. Cohen and Martin Lustig.
\newblock Very small group actions on {${\bf R}$}-trees and {D}ehn twist
  automorphisms.
\newblock {\em Topology}, 34(3):575--617, 1995.

\bibitem[CME]{canary2006fundamentals}
Richard~Douglas Canary, Albert Marden, and DBA Epstein.
\newblock {\em Fundamentals of hyperbolic manifolds: Selected expositions},
  volume 328.
\newblock Cambridge University Press, 2006.

\bibitem[CT]{CannonThurston}
James~W. Cannon and William~P. Thurston.
\newblock Group invariant {P}eano curves.
\newblock {\em Geom. Topol.}, 11:1315--1355, 2007.

\bibitem[CV]{CVouter}
Marc Culler and Karen Vogtmann.
\newblock Moduli of graphs and automorphisms of free groups.
\newblock {\em Invent. Math.}, 84(1):91--119, 1986.

\bibitem[DT1]{DT1}
Spencer Dowdall and Samuel~J Taylor.
\newblock Hyperbolic extensions of free groups.
\newblock Preprint arXiv:1406.2567, 2014.

\bibitem[DT2]{DT2}
Spencer Dowdall and Samuel~J Taylor.
\newblock The co-surface graph and the geometry of hyperbolic free group
  extensions.
\newblock In preparation, 2015.

\bibitem[FH]{FH11}
Mark Feighn and Michael Handel.
\newblock The recognition theorem for {${\rm Out}(F_n)$}.
\newblock {\em Groups Geom. Dyn.}, 5(1):39--106, 2011.

\bibitem[FM]{FMout}
Stefano Francaviglia and Armando Martino.
\newblock Metric properties of outer space.
\newblock {\em Publ. Mat.}, 55(2):433--473, 2011.

\bibitem[Ger]{Gerasimov}
Victor Gerasimov.
\newblock Floyd maps for relatively hyperbolic groups.
\newblock {\em Geom. Funct. Anal.}, 22(5):1361--1399, 2012.

\bibitem[Gui1]{Gui98}
Vincent Guirardel.
\newblock Approximations of stable actions on {${\bf R}$}-trees.
\newblock {\em Comment. Math. Helv.}, 73(1):89--121, 1998.

\bibitem[Gui2]{Gui08}
Vincent Guirardel.
\newblock Actions of finitely generated groups on {$\mathbb R$}-trees.
\newblock {\em Ann. Inst. Fourier (Grenoble)}, 58(1):159--211, 2008.

\bibitem[Ham]{hamenstadt2013boundary}
Ursula Hamenst{\"a}dt.
\newblock The boundary of the free factor graph and the free splitting graph.
\newblock Preprint arXiv:1211.1630, 2012.

\bibitem[HH]{HaHe}
Ursula Hamenstadt and Sebastian Hensel.
\newblock Convex cocompact subgroups of $\mathrm{{O}ut}({F}_n)$.
\newblock Preprint arXiv:1411.2281, 2014.

\bibitem[HM]{HMaxes}
Michael Handel and Lee Mosher.
\newblock {\em Axes in outer space.}, volume 213.
\newblock Mem. Amer. Math. Soc., 2011.

\bibitem[HV]{HVff}
Allen Hatcher and Karen Vogtmann.
\newblock The complex of free factors of a free group.
\newblock {\em Quart. J. Math.}, 49(196):459--468, 1998.

\bibitem[JKLO]{JKLO}
Woojin Jeon, Ilya Kapovich, Christopher Leininger, and Ken'ici Ohshika.
\newblock Conical limit points and the cannon-thurston map.
\newblock Preprint arXiv:1401.2638, 2014.

\bibitem[JL]{jager2008free}
Andr{\'e} J{\"a}ger and Martin Lustig.
\newblock Free group automorphisms with many fixed points at infinity.
\newblock {\em Geometry \& Topology Monographs}, 14:321--333, 2008.

\bibitem[KL1]{kapovich2007geometric}
Ilya Kapovich and Martin Lustig.
\newblock Geometric intersection number and analogues of the curve complex for
  free groups.
\newblock {\em Geom. Topol.}, 13(3):1805--1833, 2009.

\bibitem[KL2]{KLping}
Ilya Kapovich and Martin Lustig.
\newblock Ping-pong and {O}uter space.
\newblock {\em J. Topol. Anal.}, 2(02):173--201, 2010.

\bibitem[KL3]{KapLus-stabil}
Ilya Kapovich and Martin Lustig.
\newblock Stabilizers of {$\mathbb R$}-trees with free isometric actions of
  {$F_N$}.
\newblock {\em J. Group Theory}, 14(5):673--694, 2011.

\bibitem[KL4]{kapovich2014invariant}
Ilya Kapovich and Martin Lustig.
\newblock Invariant laminations for irreducible automorphisms of free groups.
\newblock {\em Q. J. Math.}, 65(4):1241--1275, 2014.

\bibitem[KL5]{KapLusCT}
Ilya Kapovich and Martin Lustig.
\newblock Cannon--{T}hurston fibers for iwip automorphisms of {$F_N$}.
\newblock {\em J. Lond. Math. Soc. (2)}, 91(1):203--224, 2015.

\bibitem[Kla]{Klar}
Erica Klarreich.
\newblock Semiconjugacies between {K}leinian group actions on the {R}iemann
  sphere.
\newblock {\em Amer. J. Math.}, 121(5):1031--1078, 1999.

\bibitem[Lev]{Levitt}
Gilbert Levitt.
\newblock Automorphisms of hyperbolic groups and graphs of groups.
\newblock {\em Geom. Dedicata}, 114:49--70, 2005.

\bibitem[LLR]{LLR}
Christopher Leininger, Darren~D. Long, and Alan~W. Reid.
\newblock Commensurators of finitely generated nonfree {K}leinian groups.
\newblock {\em Algebr. Geom. Topol.}, 11(1):605--624, 2011.

\bibitem[LMS]{LMS}
Christopher~J. Leininger, Mahan Mj, and Saul Schleimer.
\newblock The universal {C}annon-{T}hurston map and the boundary of the curve
  complex.
\newblock {\em Comment. Math. Helv.}, 86(4):769--816, 2011.

\bibitem[Man]{Mann-thesis}
Brian Mann.
\newblock {\em Some hyperbolic {O}ut({$F_N$})-graphs and nonunique ergodicity
  of very small {$F_N$}-trees}.
\newblock ProQuest LLC, Ann Arbor, MI, 2014.
\newblock Thesis (Ph.D.)--The University of Utah.

\bibitem[McM]{McM}
Curtis~T. McMullen.
\newblock Local connectivity, {K}leinian groups and geodesics on the blowup of
  the torus.
\newblock {\em Invent. Math.}, 146(1):35--91, 2001.

\bibitem[Mit1]{MitraEndingLams}
Mahan Mitra.
\newblock Ending laminations for hyperbolic group extensions.
\newblock {\em Geom. Funct. Anal.}, 7(2):379--402, 1997.

\bibitem[Mit2]{MitraCTmaps-general}
Mahan Mitra.
\newblock Cannon-{T}hurston maps for hyperbolic group extensions.
\newblock {\em Topology}, 37(3):527--538, 1998.

\bibitem[Mit3]{MitraCTmaps-trees}
Mahan Mitra.
\newblock Cannon-{T}hurston maps for trees of hyperbolic metric spaces.
\newblock {\em J. Differential Geom.}, 48(1):135--164, 1998.

\bibitem[Mit4]{Mitra98}
Mahan Mitra.
\newblock Coarse extrinsic geometry: a survey.
\newblock In {\em The {E}pstein birthday schrift}, volume~1 of {\em Geom.
  Topol. Monogr.}, pages 341--364 (electronic). Geom. Topol. Publ., Coventry,
  1998.

\bibitem[Miy]{Miyachi}
Hideki Miyachi.
\newblock Moduli of continuity of {C}annon-{T}hurston maps.
\newblock In {\em Spaces of {K}leinian groups}, volume 329 of {\em London Math.
  Soc. Lecture Note Ser.}, pages 121--149. Cambridge Univ. Press, Cambridge,
  2006.

\bibitem[Mj1]{Mj14}
Mahan Mj.
\newblock Cannon-{T}hurston maps for surface groups.
\newblock {\em Ann. of Math. (2)}, 179(1):1--80, 2014.

\bibitem[Mj2]{MjD14}
Mahan Mj.
\newblock Ending laminations and {C}annon-{T}hurston maps.
\newblock {\em Geom. Funct. Anal.}, 24(1):297--321, 2014.
\newblock With an appendix by Shubhabrata Das and Mj.

\bibitem[MP]{MP11}
Mahan Mj and Abhijit Pal.
\newblock Relative hyperbolicity, trees of spaces and {C}annon-{T}hurston maps.
\newblock {\em Geom. Dedicata}, 151:59--78, 2011.

\bibitem[MR1]{MR2}
Brian Mann and Patrick Reynolds.
\newblock {\em In preparation}.

\bibitem[MR2]{MjRafi}
Mahan Mj and Kasra Rafi.
\newblock Algebraic ending laminations and quasiconvexity.
\newblock Preprint arXiv:1506.08036v2, 2015.

\bibitem[NPR]{NPR}
Hossein Namazi, Alexandra Pettet, and Patrick Reynolds.
\newblock Ergodic decompositions for folding and unfolding paths in outer
  space.
\newblock Preprint arXiv:1410.8870, 2014.

\bibitem[Pau]{Paulin}
Fr{\'e}d{\'e}ric Paulin.
\newblock The {G}romov topology on {${\bf R}$}-trees.
\newblock {\em Topology Appl.}, 32(3):197--221, 1989.

\bibitem[Pfa]{pfaff2013out}
Catherine Pfaff.
\newblock {$Out(F_3)$} {I}ndex {R}ealization.
\newblock {\em Math. Proc. Cambridge Philos. Soc.}, 159(3):445--458, 2015.

\bibitem[Rey]{Rey12}
Patrick Reynolds.
\newblock Reducing systems for very small trees.
\newblock Preprint arXiv:1211.3378, 2012.

\bibitem[RS]{RipsSela}
Eliyahu Rips and Zlil Sela.
\newblock Cyclic splittings of finitely presented groups and the canonical
  {JSJ} decomposition.
\newblock {\em Ann. of Math. (2)}, 146(1):53--109, 1997.

\bibitem[TT]{TT}
Samuel~J. Taylor and Giulio Tiozzo.
\newblock Random extensions of free groups and surface groups are hyperbolic.
\newblock {\em Int. Math. Res. Notices}, 2015.

\end{thebibliography}

\def\cprime{$'$} \def\cprime{$'$}

\bigskip

\noindent
Department of Mathematics, Vanderbilt University \\
1326 Stevenson Center, Nashville, TN 37240, U.S.A\\
E-mail: {\tt spencer.dowdall@vanderbilt.edu}

\bigskip
\noindent
Department of Mathematics, University of Illinois at Urbana-Champaign\\
1409 W. Green Street, Urbana, IL 61801, U.S.A\\
E-mail: {\tt kapovich@math.uiuc.edu}

\bigskip
\noindent
Department of Mathematics, Yale University\\ 
10 Hillhouse Ave, New Haven, CT 06520, U.S.A\\
E-mail: {\tt s.taylor@yale.edu}

\end{document}